\documentclass[11pt]{article}
\usepackage{graphicx}
\usepackage{amsmath,amsthm,amssymb,enumerate}
\usepackage{euscript,mathrsfs}
\usepackage{color}
\usepackage{dsfont}
\usepackage[left=2cm,right=2cm,top=3.5cm,bottom=3.5cm]{geometry}
\usepackage{color}
\usepackage[framemethod=tikz]{mdframed}

\usepackage{soul}
\allowdisplaybreaks

\catcode`\@=11 \@addtoreset{equation}{section}

\catcode`\@=12

\newtheorem{Theorem}{Theorem}[section]
\newtheorem{Proposition}[Theorem]{Proposition}
\newtheorem{Lemma}[Theorem]{Lemma}
\newtheorem{Corollary}[Theorem]{Corollary}
\theoremstyle{definition}
\newtheorem{Definition}[Theorem]{Definition}

\newtheorem{Remark}[Theorem]{Remark}

\newcommand{\bTheorem}[1]{
	\begin{Theorem} \label{T#1} }
	\newcommand{\eT}{\end{Theorem}}

\newcommand{\bProposition}[1]{
	\begin{Proposition} \label{P#1}}
	\newcommand{\eP}{\end{Proposition}}

\newcommand{\bLemma}[1]{
	\begin{Lemma} \label{L#1} }
	\newcommand{\eL}{\end{Lemma}}

\newcommand{\bCorollary}[1]{
	\begin{Corollary} \label{C#1} }
	\newcommand{\eC}{\end{Corollary}}

\newcommand{\bRemark}[1]{
	\begin{Remark} \label{R#1} }
	\newcommand{\eR}{\end{Remark}}

\newcommand{\bDefinition}[1]{
	\begin{Definition} \label{D#1} }
	\newcommand{\eD}{\end{Definition}}
	
\newcommand{\Enu}{\mathcal{E}_{\mathcal{V}}}
\newcommand{\vt}{\vartheta}
\newcommand{\vr}{\varrho}
\newcommand{\vu}{\vc{u}}
\newcommand{\vm}{\vc{m}}
\newcommand{\vn}{\vc{n}}
\newcommand{\vtB}{\vt_B}
\newcommand{\vc}[1]{{\bf #1}}
\newcommand{\tvB}{\widetilde{\vB}}

\newcommand{\tvr}{\widetilde{\vr}}

\newcommand{\tvt}{\widetilde \vt}
\newcommand{\Nu}{\mathcal{V}_{t,x}}
\newcommand{\Du}{\mathbb{D}_{\vu}}
\newcommand{\DT}{\vc{D}_\vt}
\newcommand{\bB}{\vc{B}_B}
\newcommand{\vB}{\vc{B}}
\newcommand{\vU}{\vc{U}}
\newcommand{\vH}{\vc{H}}
\newcommand{\bfphi}{\boldsymbol{\varphi}}

\newcommand{\bfpsi}{\boldsymbol{\psi}}
\newcommand{\ds}{\,\mathrm{d}\sigma}
\newcommand{\wt}[1]{\widetilde{#1}}
\newcommand{\aleq}{\stackrel{<}{\sim}}
\newcommand{\tr}{{\rm Tr\,}}

\newcommand{\Ov}[1]{\overline{#1}}
\newcommand{\Curl}{{\bf curl}_x}
\newcommand{\Div}{{\rm div}_x}
\newcommand{\Grad}{\nabla_x}

\newcommand{\dx}{\,{\rm d} {x}}
\newcommand{\dt}{\,{\rm d} t }
\newcommand{\intO}[1]{\int_{\Omega} #1 \ \dx}

\newcommand{\br}{ \nonumber \\ }
\newcommand{\lang}{\left\langle}
\newcommand{\rang}{\right\rangle}

\def\softd{{\leavevmode\setbox1=\hbox{d}%
		\hbox to 1.05\wd1{d\kern-0.4ex{\char039}\hss}}}

\definecolor{Cgrey}{rgb}{0.85,0.85,0.85}
\definecolor{Cblue}{rgb}{0.50,0.85,0.85}
\definecolor{Cred}{rgb}{1,0,0}
\definecolor{fancy}{rgb}{0.10,0.85,0.10}

\begin{document}


\title{Stability of strong solutions to the full compressible magnetohydrodynamic system with non-conservative boundary conditions}
\author{Hana Mizerov\' a	}

\date{}

\maketitle

\smallskip

\centerline{Institute of Mathematics, Czech Academy of Sciences}

\centerline{\v Zitn\' a 25, CZ-115 67 Praha 1, Czech Republic}

\bigskip
\centerline{Department of Mathematical Analysis and Numerical Mathematics, Comenius University Bratislava}
\centerline{Mlynsk\' a dolina, 842 48 Bratislava, Slovakia}

\begin{abstract}
We define a dissipative measure-valued (DMV) solution to the system of equations governing the motion of a general compressible, viscous, electrically and heat conducting fluid driven by non-conservative boundary conditions. We show the stability of strong solutions to the full compressible magnetohydrodynamic system in a large class of these DMV solutions. In other words, we prove a DMV-strong uniqueness principle: a  DMV solution coincides with the strong solution emanating from the same initial data as long as the latter exists. 
\end{abstract}
\let\thefootnote\relax\footnote{The work was supported by the Czech Sciences Foundation (GA\v CR) Grant Agreement 24–11034S and by the Institute of Mathematics, Czech Academy of Sciences (RVO 67985840).}



\section{Introduction}
\label{S-I}

Magnetohydrodynamics, as a sub-discipline of fluid dynamics, describes the behaviour of electrically conducting fluids in the presence of magnetic fields. It was initiated by Hannes Alfv\' en in the late $19^{\rm th}$ century \cite{Alf} and has been further developed, in particular, in the context of  space physics, geophysics, and astrophysics where many amazing events occur due to the action of  magnetic fields, see e.g. \cite{TCD}. As an example serves the large scale dynamics of the solar convection zone which is essentially driven by the boundary conditions, see \cite{G}.

The full MHD system in one dimension has been extensively studied, see e.g. \cite{FJN,KaOka,W} and the references therein. We refer to \cite{ChZ} for the local well-posedness of strong solutions to the Cauchy problem in two dimensions. To name a few results for the three-dimensional full compressible MHD model, see \cite{HJP} for the global existence of strong solutions to the Cauchy problem assuming small data, \cite{LW} for the local existence and uniqueness of strong solutions on a bounded domain with three types of boundary conditions, see also  \cite{DuFei} for the global-in-time existence of weak solutions for any finite energy data posed on a bounded spatial domain  supplemented with conservative boundary conditions, and \cite{FeGwKwSG} for  global-in-time existence of weak solutions with non-conservative boundary conditions even for problems with large data and solutions remaining out of equilibrium in the long run. The proof of existence of a bounded absorbing set in the energy space and asymptotic compactness of trajectories for the full MHD system driven by non-conservative boundary conditions can be found in \cite{BreFeiMHD}. In \cite{H} the author proves the existence of a dissipative measure-valued solution, and also shows the DMV-strong uniqueness principle for the full MHD system, however, only for homogeneous Dirichlet  boundary conditions.

The most interesting real world phenomena can be described as dissipative systems exchanging energy and matter with their surroundings while dissipating energy in accordance with the Second law of thermodynamics. They can be mathematically described  as systems of field equations equipped with inhomogeneous boundary data. Therefore we consider the full three-dimensional compressible MHD system 
 with non-conservative  boundary conditions of both Dirichlet and Neumann type. 
 Our aim is twofold:
 \begin{itemize}
\item give a proper definition of a DMV solution to the full compressible MHD system with non-conservative boundary conditions 
\item prove stability of strong solutions to the full compressible MHD system in a class of DMV solutions. 
 \end{itemize}
 In recent years it has turned out that the concept of DMV solution together with DMV-strong uniqueness principle are powerful tools in numerical analysis of fluid flow models, cf. \cite{FLMS} and the references therein.
 
 The paper is organized as follows: in Section~\ref{S-M} we give the precise formulation of the full compressible MHD model, followed by the definition of DMV solution in Section~\ref{S-DMV}. The relative energy inequality as a key tool to prove the main result is derived in Section~\ref{S-REI}. Section~\ref{S-UP} contains several conditional and one unconditional result on stability of strong solutions in the class of DMV solutions together with the proofs. Our results with possible applications are finally summarized in Conclusions.

\section{Full compressible magnetohydrodynamic model}\label{S-M}

The time evolution of a viscous, compressible, electrically and heat conducting fluid is governed 
by the system of field equations of \emph{compressible magnetohydrodynamics}, in which $\vr = \vr(t,x)$  denotes the mass density, $\vt = \vt(t,x)$ is the (absolute) temperature,  
$\vu = \vu(t,x)$  stands for the velocity and $\vB = \vB (t,x)$ for the magnetic field:
\begin{itemize}
\item{\bf equation of continuity:}
	\begin{equation} \label{mhd_ce}
		\partial_t \vr + \Div (\vr \vu) = 0
	\end{equation}
\item{\bf momentum equation:}
	\begin{equation} \label{mhd_me}
		\partial_t (\vr \vu) + \Div (\vr \vu \otimes \vu)   + \Grad p (\vr, \vt) = \Div \mathbb{S}(\vr,\vt, \mathbb{D}(\vu)) + \Curl  \vc{B} \times \vc{B}
	\end{equation}
\item{\bf induction equation:}
\begin{equation} \label{mhd_ie} 
	\partial_t \vc{B} + \Curl (\vc{B} \times \vu ) + \Curl (\zeta (\vr,\vt) \Curl  \vc{B} ) = 0,\
	\Div \vc{B} = 0
	\end{equation}
\item{\bf internal energy balance:}
	\begin{align}
		\partial_t ( \vr e(\vr, \vt) ) + \Div (\vr e(\vr, \vt) \vu) &+ \Div \vc{q}(\vr,\vt, \Grad \vt) \br &=
		\mathbb{S}(\vr,\vt,\mathbb{D}(\vu)):\Grad \vu + \zeta(\vr,\vt) |\Curl \vB|^2  - p(\vr, \vt) \Div \vu.
		\label{mhd_ieb}
	\end{align}	
\end{itemize}
For the model see, for instance, \cite{BreFeiMHD,LL,WP}.
We suppose the pressure $p = p(\vr, \vt)$ and the internal energy $e=e(\vr, \vt)$ are interrelated through \emph{Gibbs' equation}
\begin{equation} \label{gibbs}
	\vt D s = De + p D \left( \frac{1}{\vr} \right),
\end{equation}
where $s = s(\vr, \vt)$ is the entropy. Consequently, the internal energy balance \eqref{mhd_ieb} may be reformulated
in the form of 
\begin{itemize}
\item {\bf entropy balance equation:} 
\begin{align}
	\partial_t (\vr s(\vr, \vt)) &+ \Div (\vr s(\vr, \vt) \vu) + \Div \left( \frac{\vc{q}(\vr,\vt, \Grad \vt)}{\vt} \right) \br &=
	\frac{1}{\vt} \left( \mathbb{S}(\vr,\vt, \mathbb{D}(\vu)) : \Grad \vu - \frac{\vc{q}(\vr,\vt, \Grad \vt) \cdot \Grad \vt}{\vt} + \zeta(\vr,\vt) |\Curl \vB |^2 \right). \label{mhd_eb}
\end{align}
\end{itemize}
In addition, we impose {\it hypothesis of thermodynamic stability}, 
\begin{equation} \label{hts}
	\frac{ \partial p (\vr, \vt) }{\partial \vr} > 0,\
	\frac{ \partial e (\vr, \vt) }{\partial \vt} > 0.
\end{equation}
We consider a \emph{Newtonian fluid} with the viscous stress tensor
\begin{equation} \label{S}
	\mathbb{S}(\vr,\vt, \mathbb{D}(\vu)) = \mu (\vr,\vt) \left( 2\mathbb{D}(\vu) - \frac{1}{3} \tr(\mathbb{D}(\vu)) \mathbb{I} \right) +
	\frac{\eta(\vr,\vt)}{3}\tr(\mathbb{D}(\vu)) \mathbb{I},\quad \mathbb{D}(\vu)=\frac{\Grad\vu+\Grad^t\vu}{2},
\end{equation}
and the heat flux obeying \emph{Fourier's law},
\begin{equation} \label{F}
	\vc{q}(\vr,\vt, \Grad \vt)= - \kappa (\vr,\vt) \Grad \vt.
\end{equation}	
The viscosity coefficients $\mu > 0,$ $\eta \geq 0,$ the heat conductivity coefficient $\kappa > 0,$ and the  magnetic diffusivity coefficient $\zeta >0$ are supposed to be  continuously differentiable functions of the density and temperature. 

We suppose that the fluid occupies a bounded domain $\Omega \subset R^3$ with a smooth boundary, 
\begin{equation}
\begin{aligned} 
\partial \Omega = \Gamma^\vu_D  = &\Gamma^\vt_D \cup \Gamma^\vt_N = 
\Gamma^\vB_D \cup \Gamma^\vB_N, \ \Gamma^\vu_D,  \Gamma^\vt_D, \Gamma^\vt_N, 
\Gamma^\vB_D , \Gamma^\vB_N \ \mbox{compact}, \\
& \Gamma^\vt_D \cap \Gamma^\vt_N = \emptyset,\ 
\Gamma^\vB_D \cap \Gamma^\vB_N = \emptyset.
\label{boundary}
	\end{aligned}
	\end{equation}
The sets $\Gamma^*_D$, $\Gamma^*_N$ are either empty or coincide with a finite union of connected components of $\partial \Omega$. 
The term on the right-hand side of \eqref{mhd_eb} represents the  entropy production rate, and, in accordance with the Second law of thermodynamics, it is always non-negative. 
All forms of energy  eventually transformed to heat  must be allowed to leave through the boundary. Hence, $\Gamma^\vt_D \ne \emptyset.$
Accordingly, we impose the following boundary conditions: 
\begin{itemize}
\item {\bf boundary velocity:}
		\begin{align}  
		\vu|_{\Gamma^\vu_D} &= \vu|_{\partial\Omega}=0, \label{bc_v_D} 
	\end{align}
\item {\bf boundary temperature/heat flux:}
	\begin{align}
		\vt|_{\Gamma^\vt_D} &= \vtB, \label{bc_t_D} \\ 
		\Grad \vc{q}(\vr, \vt, \Grad \vt) \cdot \vc{n}|_{\Gamma^\vt_N} &= 0,
		\label{bc_t_N}
		\end{align}
\item {\bf boundary magnetic field:} 
\begin{align} 
\vB \times \vc{n}|_{\Gamma^\vB_D} &= \vc{b}_\tau, \label{bc_m_D} \\
\vB \cdot \vc{n}|_{\Gamma^\vB_N} = b_\nu,\ 
   \Curl \vB \times \vc{n}|_{\Gamma^\vB_N} &= 0. 
 	\label{bc_m_N}
	\end{align}
\end{itemize}
Clearly, the boundary conditions prescribed on $\Gamma^{*}_D$ are of Dirichlet type, while those on 
$\Gamma^{*}_N$ are of Neumann type. The flux on $\Gamma^{\vB}_N$ is set to be zero for the sake of simplicity. We assume that the boundary data $b_\nu$, $\vc{b}_\tau$ are \emph{stationary}, see \cite[Definition 1.1]{BreFeiMHD}. It means, there exists a continuously differentiable vector field $\vc{B}_B$ such that 
\begin{align*}
\Div \bB = 0,\ \Curl \bB = 0 \ \mbox{in}\ \Omega,\ \bB \cdot \vc{n}|_{\Gamma^\vB_N} = b_\nu,\ 
		\bB \times \vc{n}|_{\Gamma^\vB_D} = \vc{b}_\tau.
\end{align*}
We suppose the boundary temperature $\vtB = \vtB(t,x)$ can be extended inside $\Omega$  and belongs to the class $C^{1,2}([0,T]\times\overline{\Omega};R^3),$ 
\begin{align}
	\vtB>0 \ \mbox{ in } \ [0,T]\times\overline{\Omega}.
\end{align}

Finally, we prescribe a suitable initial state,
\begin{align}
\vr(0,\cdot)=\vr_0, \ \vu(0,\cdot)=\vu_0,\ \vt(0,\cdot)=\vt_0, \ \vB(0,\cdot)=\vB_0,,\quad
 \vr_0,\vt_0>0,   \ \Div\vB_0=0 \mbox{ in }  \Omega.\label{ic}
\end{align}

\section{Dissipative measure-valued solutions}\label{S-DMV}

We give a precise definition of a dissipative measure-valued solution to the system of compressible magnetohydrodynamics equipped with the boundary and initial conditions as described in the previous section.

\subsection{Young measure and initial condition}

We shall consider a Young measure $\mathcal{V}$ (a parametrized family of probability measures) such that 
\begin{align}\label{YM}
	\mathcal{V} \in L^\infty_{\rm weak-*}((0,T)\times\Omega;\mathcal{P}(\mathcal{F})),\ \mathcal{V}=\{\Nu\}_{(t,x)\in(0,T)\times\Omega},
	\end{align}
	where 
\begin{align}\label{PS}
		\mathcal{F}=\bigg\{[\vr,\vu,\vt,\vB,\Du, \vc{D}_{\vt},\vc{C}_{\vB}] \; \big | \; \vr ,\vt \in [0,\infty),\; \vu, \vB , \vc{D}_\vt ,  \vc{C}_{\vB} \in R^3, \ \Du\in R^{3\times 3}_{\rm sym} \bigg\}
\end{align}
is a suitable phase space for the system \eqref{mhd_ce} - \eqref{mhd_ie}, \eqref{mhd_eb}. 
Its initial state \eqref{ic}
can be expressed by means of a parametrized probability measure 
\begin{align}
\mathcal{V}_0=\{\mathcal{V}_{0,x}\}_{x\in\Omega} \in L^\infty_{\rm weak-*}(\Omega;\mathcal{P}(\mathcal{F}_0))
\label{YM_IC}
\end{align}
acting on the phase space 
\begin{align}\label{YM_IC_PS}
\mathcal{F}_0=\bigg\{[\vr_0,\vu_0,\vt_0,\vB_0] \big | \; \vr_0 ,\vt_0 \in [0,\infty),\; \vu_0, \vB_0 \in R^3  \bigg\}.
\end{align} 

\subsection{Definition of a dissipative measure-valued solution}

	\begin{Definition}[{\bf Dissipative measure-valued solution}]\label{D:mvs}
	We say that a Young measure $\mathcal{V}$ is a dissipative  measure-valued solution to the compressible magnetohydrodynamic system \eqref{mhd_ce}-\eqref{mhd_ie}, \eqref{mhd_eb} with \eqref{S}, \eqref{F}, \eqref{bc_v_D} -   \eqref{bc_m_N} if \eqref{YM} holds true, the initial condition is given by \eqref{YM_IC}, and the following holds true:
\begin{itemize}
	\item{\bf equation of continuity:}
	\begin{align} 	\label{mvs_ce}
	\left[ \intO{\lang \Nu;\vr\rang\varphi}\right]_{t=0}^{t=\tau}=\int_0^\tau \intO{ \Big( \lang \Nu;\vr\rang \partial_t \varphi + \lang \Nu; \vr\vu \rang \cdot \Grad \varphi \Big) } \dt 
\end{align}
for any $\varphi \in C^{1}([0, T]\times \Ov{\Omega})$ and for a.e. $0 < \tau < T.$

\item{\bf momentum equation:}
\begin{equation}	\label{mvs_me}
\begin{aligned}
	&\left[\intO{\lang \Nu; \vr\vu\rang \bfphi}\right]_{t=0}^{t=\tau} =\\
	&= \int_0^\tau \intO{ \Big( \lang \Nu; \vr \vu\rang \cdot \partial_t \bfphi - \lang \Nu; \mathbb{S}(\vr,\vt, \Du)\rang: \Grad \bfphi \Big) }\dt \\
		&+\int_0^\tau \intO{\Big( \lang \Nu; \vr \vu \otimes \vu -\vB \otimes \vB \rang : \Grad \bfphi +\lang \Nu;  p(\vr,\vt)+ \frac{1}{2} |\vB |^2 \rang  \Div \bfphi \Big)}\dt   \\
	&
+	\int_0^\tau \int_{\Ov{\Omega}}{ \Grad\bfphi: {\rm d}\mathfrak{R}^M \dx } \dt  +	\int_0^\tau \int_{\Ov{\Omega}}{ \Grad\bfphi: {\rm d}\mathfrak{R}^B\dx} \dt  .
\end{aligned}
\end{equation}
 for any $\bfphi \in C^1([0,T] \times \Ov{\Omega}; R^3),$ $\bfphi|_{\Gamma_D^\vu}=0,$ and for a.e. $0 \leq \tau \leq T.$ \\
Here  $\mathfrak{R}^M,$ $\mathfrak{R}^B \in \mathcal{M}([0,T]\times\Ov{\Omega})$ are called  {\it concentration defect measures} and  stem from the nonlinear terms in the momentum equation, $\vr\vu\otimes\vu+p\mathbb{I}$ and $\frac{1}{2}|\vB|^2\mathbb{I}-\vB\otimes\vB,$ respectively.

\item {\bf magnetic induction equation:}
\begin{equation}\label{mvs_ie}
\begin{aligned}
	&\left[ \intO{\lang \Nu;\vB\rang \cdot \bfphi}  \right]_{t=0}^{t=\tau} = \\
	&=\int_0^\tau \intO{ \Big( \lang \Nu; \vB\rang \cdot \partial_t \bfphi - \lang \Nu; \vB \times \vu + \zeta (\vr,\vt) \vc{C}_B  \rang \cdot \Curl \bfphi \Big) } \dt  
\end{aligned}
\end{equation}	
for any $\bfphi \in C^1([0,T] \times \Ov{\Omega}),$ $\Div\bfphi=0,$  $\bfphi\times\vn|_{\Gamma_D^\vB}=0,$ $\bfphi\cdot\vn|_{\Gamma_N^\vB}=0,$ and for a.a. $0 \leq \tau \leq T.$
Moreover, 
\begin{align}\label{mvs_me_div}
\int_0^\tau\intO{\lang \Nu; \vB \rang\cdot\varphi}\dt=0
\end{align}
for any $\varphi \in C_c^1([0,T]\times\Ov{\Omega})$  and for a.a. $0 \leq \tau \leq T.$

\item {\bf entropy inequality:}
\begin{equation}
\begin{aligned}
&\left[ \intO{ \lang \Nu;\vr s(\vr,\vt)\rang \cdot \varphi } \right]_{t=0}^{t=\tau}  -	\\
&- \int_0^\tau \intO{ \lang \Nu; \frac{1}{\vt}
		\left( \mathbb{S}(\vr,\vt, \Du) : \Du  +  \frac{\kappa(\vr,\vt)|\vc{D}_{\vt}|^2 }{\vt} + \zeta (\vr,\vt) |\vc{C}_{\vB}|^2 \right)\rang\varphi }	\dt	 \\
		& \geq \int_0^\tau \intO{ \Big(\lang \Nu;\vr s(\vr,\vt)\rang  \partial_t \varphi + \lang \Nu; \vr s(\vr,\vt) \vu - \frac{\kappa(\vr,\vt) \vc{D}_{\vt}}{\vt} \rang \cdot \Grad \varphi \Big) } \dt 
	\label{mvs_ei}	
\end{aligned}
\end{equation}
 for any $\varphi \in C^1_c(([0,T] \times \Ov{\Omega})$, $\varphi \geq 0$, $\varphi|_{\Gamma_D^\vt}=0,$ and for  a.a. $0 \leq \tau \leq T.$

\item {\bf ballistic energy inequality:}
\begin{equation}
\begin{aligned}
	& \left[\intO{ \lang \Nu; \left( \frac{1}{2} \vr |\vu|^2 + \vr e(\vr, \vt) + \frac{1}{2} |\vB |^2 - \tvt \vr s(\vr, \vt) - \tvB\cdot \vB\right) \rang} \right]_{t=0}^{t=\tau} +\mathcal{D}_{\tvt}(\tau) +\mathcal{D}_{\tvB}(\tau)+\\
&	+ \int_0^\tau \intO{ \lang	\Nu;\frac{1}{\vt} \left( \mathbb{S}(\vr,\vt, \Du) : \Du +\frac{\kappa(\vr,\vt)| \vc{D}_{\vt}|^2}{\vt} + \zeta (\vr,\vt) |\vc{C}_{\vB} |^2 \right) \rang \tvt} \dt  \\
&\leq -\int_0^\tau \intO{\Big(\lang \Nu;   \vr s (\vr, \vt) \rang \partial_t \tvt +\lang \Nu; \vr s (\vr, \vt) \vu  -\frac{\kappa(\vr,\vt)\vc{D}_{\vt}}{ \vt}\rang \cdot \Grad \tvt                       \Big)} \dt\\ 
	& -\int_0^\tau  \intO{ \Big( \lang \Nu; \vB \rang \cdot \partial_t \tvB -  \lang \Nu; \vB \times \vu + \zeta (\vr,\vt) \vc{C}_{\vB} \rang \cdot \Curl \tvB\Big) }\dt .
		\label{mvs_bei}
\end{aligned}
\end{equation}
 for any 
$\tvt \in C^{1,2}([0,T] \times \Ov{\Omega})$, $\tvt > 0$, $\tvt|_{\Gamma_D^\vt} = \vtB$, and for any $\tvB \in  C^{1,2}([0,T] \times \Ov{\Omega};R^3),$ $\Div \tvB=0,$  $\tvB \times \vn|_{\Gamma_D^\vB}=\vc{b}_\tau,$ $\tvB \cdot \vn|_{\Gamma_N^\vB}=b_\nu,$ and for a.a. $0 \leq \tau \leq T.$\\ Here  $\mathcal{D}_{\tvt}, \mathcal{D}_{\tvB} \in L^\infty(0,T),$ $\mathcal{D}_{\tvt}\geq 0,$ $\mathcal{D}_{\tvB} \geq 0,$ are called {\it dissipation defects.} 
\item {\bf concentration vs. dissipation defects:}
\begin{align}
\int_0^\tau \xi(t)\int_{\Ov{\Omega}} {\rm d}|\mathfrak{R}^M| &\aleq \int_0^\tau \xi(t)\mathcal{D}_{\tvt}(t) \dt \label{cd_m}
\end{align}
for any $\tvt \in C^{1,2}([0,T] \times \Ov{\Omega})$, $\tvt > 0$, $\tvt|_{\Gamma_D^\vt} = \vtB$, and
for a.a. $0 \leq \tau \leq T$ and any $\xi \in C([0,T]),$ $\xi(t)\geq 0,$
\begin{align}
\int_0^\tau \chi(t)\int_{\Ov{\Omega}} {\rm d}|\mathfrak{R}^B| &\aleq \int_0^\tau \chi(t)\mathcal{D}_{\tvB}(t) \dt \label{cd_b}
\end{align}
for any $\tvB \in  C^{1,2}([0,T] \times \Ov{\Omega};R^3),$ $\Div \tvB=0,$ $\tvB \times \vn|_{\Gamma_D^\vB}=\vc{b}_\tau$ and $\tvB \cdot \vn|_{\Gamma_N^\vB}=b_\nu,$ and
for a.a. $0 \leq \tau \leq T$ and any $\chi \in C([0,T]),$ $\chi(t)\geq 0.$
\item {\bf compatibility conditions:}\\
For any $\mathbb Z\in C^1([0,T]\times\Ov{\Omega}; R^{3 \times 3}_{\rm sym})$  it holds that 
\begin{align}\label{cc_S}
&-\int_0^T\intO{  \lang \Nu; \vu\rang\cdot\Div\mathbb Z}\dt =\int_0^T \intO{  \lang \Nu; \Du\rang:\mathbb{Z}}\dt .
\end{align}
For any $\bfpsi\in C^1([0,T]\times\Ov{\Omega};R^3),$ $\bfpsi\cdot\vn|_{\Gamma_N^\vt}=0$ and $\tvt\in C^{1,2}([0,T]\times\Ov{\Omega})$ with $\tvt|_{\Gamma_D^\vt}=\vt_B,$ it holds that
\begin{align}\label{cc_t}
 -\int_0^T \intO{ \lang\Nu;\vt-\tvt\rang\Div\bfpsi}\dt= \int_0^T\intO{ \lang\Nu;\vc{D}_{\vt}-\Grad\tvt\rang\cdot\bfpsi}\dt .
 \end{align}
For any $\wt{\vB}\in C^{1,2}([0,T]\times\Ov{\Omega};R^3),$ $\wt{\vB}\times\vn|_{\Gamma_D^\vB}=\vc{b}_\tau$, and any $\vc{G}\in C^{1}([0,T]\times\Ov{\Omega};R^3),$  $\vc{G}\times\vn|_{\Gamma_N^\vB}=0$ or $\vc{G}|_{\Gamma_N^\vB}=0,$
it holds that 
\begin{equation}
\begin{aligned}\label{cc_B}
&\int_0^T  \intO{    \lang \Nu; \vB -\wt{\vB} \rang \cdot  \Curl\vc{G}}\dt =\int_0^T  \intO{  \vc{G} \cdot \lang \Nu; \vc{C}_\vB-\Curl\wt{\vB} \rang  }\dt .
\end{aligned}
\end{equation}

\item {\bf Korn-Poincar\' e inequality:}\\
For any $\wt{\vu}\in L^2(0,T;W^{1,2}_0(\Omega;R^3))$ it holds that 
\begin{align}\label{KP}
\int_0^\tau\intO{\lang\Nu;|\vu-\wt{\vu}|^2\rang}\dt \aleq \int_0^\tau\intO{\lang\Nu;|\mathbb{T}[\mathbb{D}_\vu]-\mathbb{T}[\mathbb{D}(\wt{\vu})]|^2\rang}\dt ,
\end{align}
where 
\begin{align*}
\mathbb{T}[\mathbb{D}_\vu]=\mathbb{D}_\vu-\frac{1}{3}\tr{\left(\mathbb{D}_\vu\right)}, \ \mathbb{T}[\mathbb{D}(\wt{\vu})]=\mathbb{D}(\wt{\vu})-\frac{1}{3}\tr{\left(\mathbb{D}(\wt{\vu})\right)}.
\end{align*}
\end{itemize}				
		\end{Definition}
\begin{Remark}
The Korn-Poincar\' e inequality \eqref{KP} proven for a Young measure generated by a sequence of $W^{1,2}_0(\Omega;R^3)-$functions can be found in \cite[Lemma 2.2]{FeiNoSL}.

\end{Remark}

\section{Relative energy inequality}\label{S-REI}

The relative energy inequality, introduced by Dafermos \cite{Daf}, is a powerful and elegant tool for measuring the stability of a strong solution compared to a generalized solution, i.e. for proving weak-strong uniqueness principle. This method was developed by Feireisl, Novotn\' y and collaborators, see e.g. \cite{FeiJNo,FeiNoSL,FeiNoWS,FGSGW,FLMS} and the references therein. 

We consider a suitable relative energy  for the compressible MHD system,
\begin{equation}
\begin{aligned}\label{re_fun}
&E\left(\vr, \vt, \vu, \vB \; | \; r, \Theta, \vU, \vH \right) = \frac{1}{2}\vr|\vu-\vU|^2+\frac{1}{2}|\vB-\vH|^2 +\\
&+ \vr e (\vr,\vt)-\Theta\left( \vr s(\vr,\vt)-r s(r,\Theta)\right)-\left(e(r,\Theta) - \Theta s(r,\Theta) +\frac{p(r,\Theta)}{r}  \right)(\vr-r) - r e (r,\Theta),
\end{aligned}
\end{equation}
see also \cite{FeGwKwSG}. The variables $r, \Theta, \vU, \vH$ stand for arbitrary sufficiently smooth functions satisfying the compatibility conditions
\begin{equation}
\begin{aligned}
	r &> 0,\ \Theta > 0 \ \mbox{in}\ [0,T] \times \Ov{\Omega}, \ \Theta|_{\Gamma_D^\vt} = \vtB, \ \vU|_{\partial\Omega}=0, \\
	\Div \vH &= 0 \ \mbox{in}\ (0,T) \times \Omega,\
	\vH \times \vn|_{\Gamma_D^\vB} =
	\bB \times \vn|_{\Gamma_D^\vB} = {\bf b_\tau},\ 	\vH \cdot \vn|_{\Gamma_N^\vB} =
	\bB \cdot \vn|_{\Gamma_N^\vB}=b_{\nu}.
\label{smooth_cc}
\end{aligned}
\end{equation}
Hypothesis of thermodynamic stability  \eqref{hts} allows to show that the energy
\begin{align*}
E(\vr, \vt, \vu, \vB) = \frac{1}{2} \vr |\vu|^2 + \vr e(\vr, \vt) + \frac{1}{2} |\vB|^2
\end{align*}
 is a strictly convex l.s.c. function when expressed in the variables $\vr, S = \vr s(\vr, \vt), \vm = \vr \vu, \vB,$ see \cite[Chapter 3, Section 3.1]{FeiNoOS} for details. Since, 
\begin{align*}
\frac{\partial E(\vr, S, \vm, \vB)}{\partial \vr} = \left(  e(\vr, \vt) - \vt s(\vr, \vt) + \frac{p(\vr, \vt) }{\vr} \right),\ \frac{\partial E(\vr, S, \vm, \vB)}{\partial S} = \vt,
\end{align*}
where $\left(  e(\vr, \vt) - \vt s(\vr, \vt) + \frac{p(\vr, \vt) }{\vr} \right)$ is Gibbs' free energy,  the relative energy can be seen as Bregman distance associated to the convex function $E$.
In particular,
\begin{align*}
E \left( \vr, \vt, \vu, \vB \ \Big| r, \Theta, \vU, \vH\right) \geq 0\ \mbox{and} \
E \left( \vr, \vt, \vu, \vB \ \Big| r, \Theta, \vU, \vH \right)  = 0 \ \Leftrightarrow \
( \vr, \vt, \vu, \vB) = (r, \Theta, \vU, \vH )
\end{align*}
as long as $r > 0$.

Given a dissipative measure-valued solution $\mathcal{V}$ to the compressible MHD system \eqref{mhd_ce}-\eqref{mhd_ie}, \eqref{mhd_eb} according to Definition~\ref{D:mvs} we define 
\begin{align*}
\Enu(\tau) = \intO{\lang \mathcal{V}_{\tau,x}; E(\vr,\vt,\vu,\vB \; | \; r,\Theta,\vU,\vH )\rang}. 
\end{align*}
A straightforward manipulation yields
\begin{align*}
\Enu(\tau) &= \intO{\lang  \mathcal{V}_{\tau,x}; \frac{1}{2} \vr |\vu|^2 + \vr e(\vr, \vt) + \frac{1}{2} |\vB|^2 -\Theta \vr s(\vr,\vt)- \vB\cdot\vH  \rang}+\\
& + \frac{1}{2}\intO{ |\vH|^2(\tau,\cdot) }+ \frac{1}{2}\intO{\lang  \mathcal{V}_{\tau,x};\vr \rang  |\vU|^2}-\intO{\lang  \mathcal{V}_{\tau,x}; \vr\vu\rang\cdot \vU}-\\
&- \intO{\lang  \mathcal{V}_{\tau,x}; \vr \rang \left( e(r,\Theta)-\Theta s (r,\Theta)+\frac{p(r,\Theta)}{r}\right)} +\intO{p(r,\Theta)(\tau,\cdot)}.
\end{align*} 
  In order to estimate the terms on the right-hand side we start with the ballistic energy inequality \eqref{mvs_bei}, then set $\varphi=\frac{1}{2}|\vU|^2- \left(  e(r,\Theta)-\Theta s (r,\Theta)+\frac{p(r,\Theta)}{r}\right)$ in the continuity equation \eqref{mvs_ce}, and $\bfphi=-\vU$ in the momentum equation \eqref{mvs_me}. We sum them up and add the identity
 \begin{align*}
\frac 1 2 \intO{  |\vH|^2(\tau,\cdot)}+\intO{ p(r,\Theta)(\tau,\cdot)}&=\frac 1 2\intO{ |\vH|^2(0,\cdot)}  +\intO{ p(r,\Theta)(0,\cdot)}+\\
  &+ \int_0^\tau\intO{\vH\cdot\partial_t\vH}\dt+ \int_0^\tau\intO{\partial_t p(r,\Theta)}\dt.
 \end{align*}
This, combined with Gibb's equation \eqref{gibbs}, leads to
\begin{equation}
\begin{aligned}
&\bigg[\Enu(t)\bigg]_{t=0}^{t=\tau} +\mathcal{D}_{\Theta}(\tau) +\mathcal{D}_{\vH}(\tau)+\\
&+ \int_0^\tau \intO{ \lang	\Nu;\frac{1}{\vt} \left( \mathbb{S}(\vr,\vt, \Du) : \Du +\frac{\kappa(\vr,\vt)| \vc{D}_{\vt}|^2}{\vt} + \zeta (\vr,\vt) |\vc{C}_{\vB} |^2 \right) \rang \Theta} \dt +\\
 &\leq-\int_0^\tau \intO{\Big(\lang \Nu;   \vr s (\vr, \vt) \rang \partial_t \Theta +\lang \Nu; \vr s (\vr, \vt) \vu  -\frac{\kappa(\vr,\vt)\vc{D}_{\vt}}{ \vt}\rang \cdot \Grad \Theta                       \Big)} \dt+\\ 
& +\int_0^\tau \intO{\Big(\lang \Nu;   \vr  \rang s(r,\Theta)\partial_t \Theta +\lang \Nu; \vr  \vu \rang \cdot \Grad \Theta s(r,\Theta)  \Big)} \dt -&\\
	& -\int_0^\tau  \intO{ \Big( \lang \Nu; \vB \rang \cdot \partial_t \vH -  \lang \Nu; (\vB \times \vu) + \zeta (\vr,\vt) \vc{C}_{\vB} \rang \cdot \Curl \vH \Big) }\dt+ \int_0^\tau \intO{ \vH \cdot \partial_t \vH   } \dt +\\
	&+\int_0^\tau  \intO{\lang \Nu;(r-\vr)\rang   \frac{\partial_t p(r,\Theta)}{r} } -\int_0^\tau  \intO{\lang \Nu;\vr\vu\rang\cdot \left(  \frac{\Grad p(r,\Theta)}{r}\right) } -\\
		&- \int_0^\tau \intO{  \lang \Nu; \vr( \vu-\vU)\rang \cdot \Big(\partial_t\vU   +\vU\cdot\Grad\vU+\frac{1}{r}\Grad p(r,\Theta)\Big)}\dt +\\
		&+ \int_0^\tau \intO{  \lang \Nu; \vr(\vu-\vU)\rang \frac{1}{r}\Grad p(r,\Theta)}\dt+  \int_0^\tau \intO{  \lang \Nu; \mathbb{S}(\vr,\vt, \Du)\rang: \Grad \vU  }\dt  -\\
		&-\int_0^\tau \intO{\lang \Nu; \vr (\vu-\vU) \otimes (\vu-\vU) -\vB \otimes \vB + p(\vr,\vt)\mathbb{I}+ \frac{1}{2} |\vB |^2\mathbb{I} \rang : \Grad \vU }\dt-   \\
&-	\int_0^\tau \int_{\Ov{\Omega}}{ \Grad\vU: {\rm d}\mathfrak{R}^M \dx } \dt  -	\int_0^\tau \int_{\Ov{\Omega}}{ \Grad\vU: {\rm d}\mathfrak{R}^B\dx} \dt  
\end{aligned}
\end{equation}
From now on we suppose $(r,\Theta,\vU,\vH)$ is a strong solution to the full compressible MHD system given in Section~\ref{S-M} satisfying \eqref{smooth_cc} and emanating from the same initial data as DMV solution defined in Section~\ref{S-DMV}. This yields
\begin{align*}
-&\int_0^\tau  \intO{ \lang \Nu; \vr(\vu-\vU)\rang\cdot \left( \partial_t\vU + \vU\cdot\Grad\vU+\frac{1}{r}\Grad p(r,\Theta)\right)} \dt = \\
&=\int_0^\tau  \intO{ \lang \Nu;  \left(\frac{\vr}{r}-1\right)(\vU-\vu)\rang \cdot \left( \Div\mathbb{S}(r,\Theta,\mathbb{D}(\vU))+\Curl \vH \times \vH \right) }\dt + \\
&+\int_0^\tau  \intO{  \lang \Nu; (\vU- \vu)\rang\cdot\left(\Div\mathbb{S}(r,\Theta,\mathbb{D}(\vU))+\Curl \vH \times \vH \right) }\dt,
\end{align*}
and the above inequality reduces to 
\begin{equation}
\begin{aligned}
&\Enu(\tau)+\mathcal{D}_{\Theta}(\tau) +\mathcal{D}_{\vH}(\tau) 	+\\
&+ \int_0^\tau \intO{ \lang	\Nu;\frac{1}{\vt} \left( \mathbb{S}(\vr,\vt, \Du) : \Du +\frac{\kappa(\vr,\vt)| \vc{D}_{\vt}|^2}{\vt} + \zeta (\vr,\vt) |\vc{C}_{\vB} |^2 \right) \rang \Theta} \dt+  \\
&\leq -\int_0^\tau \intO{\Big(\lang \Nu;   \vr s (\vr, \vt) \rang \partial_t \Theta +\lang \Nu; \vr s (\vr, \vt) \vu  -\frac{\kappa(\vr,\vt)\vc{D}_{\vt}}{ \vt}\rang \cdot \Grad \Theta                       \Big)} \dt+\\ 
& +\int_0^\tau \intO{\Big(\lang \Nu;   \vr  \rang s(r,\Theta)\partial_t \Theta +\lang \Nu; \vr  \vu \rang \cdot \Grad \Theta s(r,\Theta)  \Big)} \dt +&\\
	&+\int_0^\tau  \intO{\lang \Nu;(r-\vr)\rang   \frac{\partial_t p(r,\Theta)}{r} } -\int_0^\tau  \intO{\lang \Nu;\vr\vu\rang\cdot \left(  \frac{\Grad p(r,\Theta)}{r}\right) } +\\
	&+\int_0^\tau  \intO{ \lang \Nu;  \left(\frac{\vr}{r}-1\right)(\vU-\vu)\rang\left( \Div\mathbb{S}(r\Theta,\mathbb{D}(\vU))+\Curl \vH \times \vH \right) }\dt + \\
&+\int_0^\tau  \intO{  \lang \Nu; (\vU- \vu)\rang\left(\Div\mathbb{S}(r,\Theta,\mathbb{D}(\vU))+\Curl \vH \times \vH \right) }\dt+\\
		&+ \int_0^\tau \intO{  \lang \Nu; \tvr(\vu-\vU)\rang \frac{1}{r}\Grad p(r,\Theta)}\dt + \int_0^\tau \intO{  \lang \Nu; \mathbb{S}(\vr,\vt, \Du)\rang: \Grad \vU  }\dt -\\
		&-\int_0^\tau \intO{\Big( \lang \Nu; \vr (\vu-\vU) \otimes (\vu-\vU) -\vB \otimes \vB \rang : \Grad \vU +\lang \Nu;  p(\vr,\vt)+ \frac{1}{2} |\vB |^2 \rang  \Div \vU \Big)}\dt-   \\
		&-\int_0^\tau  \intO{ \Big( \lang \Nu; \vB \rang \cdot \partial_t \vH -  \lang \Nu; (\vB \times \vu) + \zeta (\vr,\vt) \vc{C}_{\vB} \rang \cdot \Curl \vH \Big) }\dt+\int_0^\tau \intO{ \vH \cdot \partial_t \vH   } \dt-\\
&-\int_0^\tau \int_{\Ov{\Omega}}{ \Grad\vU: {\rm d}\mathfrak{R}^M \dx } \dt  -	\int_0^\tau \int_{\Ov{\Omega}}{ \Grad\vU: {\rm d}\mathfrak{R}^B\dx} \dt  .
\end{aligned}
\end{equation}
From the exact induction equation  we get the identity 
\begin{align*}
-&\int_0^\tau  \intO{ \Big( \lang \Nu; \vB \rang \cdot \partial_t \vH -  \lang \Nu; (\vB \times \vu) + \zeta (\vr,\vt) \vc{C}_{\vB} \rang \cdot \Curl \vH \Big) }\dt+\int_0^\tau \intO{ \vH \cdot \partial_t \vH   } \dt = \\
&=\int_0^\tau  \intO{    \lang \Nu; \vB -\vH \rang \cdot \left(\Curl(\vH\times\vU) + \Curl(\zeta (r,\Theta) \Curl\vH)\right) }\dt+\\
&+\int_0^\tau  \intO{  \lang \Nu; (\vB \times \vu) + \zeta (\vr,\vt) \vc{C}_{\vB} \rang \cdot \Curl \vH  }\dt.
\end{align*}
Further computation implies  
\begin{align*}
&\int_0^\tau  \intO{    \lang \Nu; \vB -\vH \rang \cdot \Curl(\vH\times\vU) }\dt\\
&= \int_0^\tau  \intO{    \lang \Nu; \vB  \rang \cdot \Curl(\vH\times\vU) }\dt - \int_0^\tau  \intO{    \lang \Nu; \vH \rang \cdot \Curl(\vH\times\vU) }\dt=\\
&=\int_0^\tau  \intO{    \lang \Nu; \vB  \rang \cdot \left(\vH \Div\vU -\vU\Div\vH + (\vU\cdot\nabla)\vH -(\vH\cdot\nabla)\vU\right) }\dt-\\
&- \int_0^\tau  \intO{    \lang \Nu;\Curl\vH \cdot (\vH\times\vU) \rang}\dt+\int_0^\tau  \int_{\partial\Omega}{\vH\times(\vH\times\vU) \cdot \vn }\ds\dt\\
&=\int_0^\tau  \intO{    \lang \Nu;   \left(\vH \cdot\vB\mathbb I -\vB\otimes\vH-\vH\otimes\vB \right)  \rang :\Grad\vU }\dt+\\
&+ \int_0^\tau  \intO{    \lang \Nu;  (\vU\times\vB)\cdot\Curl\vH \rang}\dt-\\
&- \int_0^\tau  \intO{    \lang \Nu;(\Curl\vH \times \vH) \cdot\vU \rang}\dt+\int_0^\tau  \int_{\partial\Omega}{\vH\times(\vH\times\vU) \cdot \vn }\ds\dt\\
&= \int_0^\tau  \intO{    \lang \Nu;   \left(\vH\otimes\vH -\frac{1}{2}|\vH|^2\mathbb{I}+\vH \cdot\vB\mathbb I -\vB\otimes\vH-\vH\otimes\vB \right)\rang :\Grad\vU }\dt+\\
&+\int_0^\tau  \intO{    \lang \Nu;  (\vU\times\vB)\cdot\Curl\vH \rang}\dt + \int_0^\tau  \intO{    \lang \Nu;\Div\vH(\vH\cdot\vU)\rang}\dt +\\
&-\int_0^\tau  \int_{\partial\Omega}{(\vH\cdot\vn)(\vH\cdot\vU)-\frac{1}{2}|\vH|^2(\vU\cdot\vn)}\ds\dt +\int_0^\tau  \int_{\partial\Omega}{\vH\times(\vH\times\vU) \cdot \vn }\ds\dt\\
&= \int_0^\tau  \intO{    \lang \Nu;  \left(\vH\otimes\vH -\frac{1}{2}|\vH|^2\mathbb{I}+\vH \cdot\vB\mathbb I -\vB\otimes\vH-\vH\otimes\vB \right) \rang:\Grad\vU }\dt+\\
&+\int_0^\tau  \intO{    \lang \Nu;  (\vU\times\vB)\cdot\Curl\vH \rang}\dt-\frac{1}{2}\int_0^\tau  \int_{\partial\Omega}{|\vH|^2(\vU\cdot\vn)}\ds\dt\\
&= \int_0^\tau  \intO{    \lang \Nu;    \left(\vH\otimes\vH -\frac{1}{2}|\vH|^2\mathbb{I}+\vH \cdot\vB\mathbb I -\vB\otimes\vH-\vH\otimes\vB \right)\rang :\Grad\vU }\dt+\\
&+ \int_0^\tau  \intO{    \lang \Nu;  (\vU\times\vB)\cdot\Curl\vH \rang}\dt,
\end{align*}
where we have used $\Div\vH=0$ and the homogeneous Dirichlet boundary condition \eqref{bc_v_D} for velocity.
Hence, 
\begin{align*}
&\int_0^\tau  \intO{    \lang \Nu; \vB \otimes\vB -\frac{1}{2}|\vB|^2\mathbb{I}\rang:\Grad\vU }\dt +\int_0^\tau  \intO{    \lang \Nu; \vU-\vu \rang\cdot(\Curl\vH\times\vH)}\dt+\\
&+\int_0^\tau  \intO{  \lang \Nu; (\vB \times \vu) \rang \cdot \Curl \vH  }\dt+\int_0^\tau  \intO{    \lang \Nu;  (\vU\times\vB)\cdot\Curl\vH \rang}\dt+\\
&+\int_0^\tau  \intO{    \lang \Nu; \left(\vH\otimes\vH -\frac{1}{2}|\vH|^2\mathbb{I}+\vH \cdot\vB\mathbb I -\vB\otimes\vH-\vH\otimes\vB \right)\rang :\Grad\vU }\dt \\
&= \int_0^\tau  \intO{    \lang \Nu; (\vB-\vH)\cdot\Grad\vU\cdot(\vB-\vH) - \frac{1}{2}\Div\vU|\vH-\vB|^2\rang}\dt-\\
&-\int_0^\tau  \intO{    \lang \Nu; (\vB-\vH)\times(\vU-\vu)\rang\cdot\Curl\vH}\dt.
\end{align*}
In addition, by \eqref{cc_B}, we have 
\begin{align*}
&\int_0^\tau  \intO{  \lang \Nu;  \zeta (\vr,\vt) \vc{C}_{\vB} \rang \cdot \Curl \vH  }\dt+\int_0^\tau  \intO{    \lang \Nu; \vB -\vH \rang \cdot \Curl(\zeta (r,\Theta) \Curl\vH) }\dt=\\
& =\int_0^\tau  \intO{  \lang \Nu;  \zeta (\vr,\vt) \vc{C}_{\vB} \rang \cdot \Curl \vH  }\dt+\int_0^\tau  \intO{  \zeta (r,\Theta) \Curl\vH\cdot \lang \Nu; \vc{C}_\vB -\Curl\vH \rang }\dt.
\end{align*}
The relative energy inequality now takes the form 
\begin{equation}\label{rei_1}
\begin{aligned}
&\Enu(\tau)+\mathcal{D}_{\Theta}(\tau) +\mathcal{D}_{\vH}(\tau)+\\
& +\int_0^\tau \intO{ \lang	\Nu;\frac{1}{\vt} \mathbb{S}(\vr,\vt, \Du) : \Du +\frac{\kappa(\vr,\vt)}{\vt}\vc{D}_{\vt}\left( \frac{\vc{D}_{\vt}}{\vt}-\frac{\Grad\Theta}{\Theta}\right) + \frac{\zeta (\vr,\vt)}{\vt} |\vc{C}_{\vB} |^2  \rang \Theta} \dt \\
&\leq  \int_0^\tau \intO{  \lang \Nu; \mathbb{S}(\vr,\vt, \Du)\rang: \Grad \vU  }\dt + \int_0^\tau  \intO{  \lang \Nu; (\vU- \vu)\rang\cdot\Div\mathbb{S}(r,\Theta,\mathbb{D}(\vU)) }\dt +\\
&+ \int_0^\tau  \intO{\lang \Nu;\left(1-\frac{\vr}{r}\right)\rang   \left(\partial_t p(r,\Theta)+\vU\cdot\Grad p(r,\Theta)\right) } -\\
& -\int_0^\tau \intO{r \lang \Nu;    s (\vr, \vt)-s(r,\Theta) \rang  \Big(\partial_t \Theta +\vU\cdot\Grad\Theta \Big)} \dt -\\
&-\int_0^\tau \intO{\lang \Nu;\vU\rang\cdot\Grad p(r,\Theta)}-\int_0^\tau \intO{\lang \Nu;  p(\vr,\vt)\rang  \Div \vU }+\\
&+\int_0^\tau  \intO{   \zeta (r,\Theta) \Curl\vH \cdot \lang \Nu; \vc{C}_\vB-\Curl\vH \rang  }\dt+\int_0^\tau  \intO{   \lang \Nu;  \zeta (\vr,\vt) \vc{C}_{\vB} \rang \cdot \Curl \vH }\dt-\\
&-	\int_0^\tau \int_{\Ov{\Omega}}{ \Grad\vU: {\rm d}\mathfrak{R}^M \dx } \dt  -	\int_0^\tau \int_{\Ov{\Omega}}{ \Grad\vU: {\rm d}\mathfrak{R}^B\dx} \dt  + R_1,
\end{aligned}
\end{equation}
where $R_1$ is given by
\begin{equation}\label{R1}
\begin{aligned}
R_1&=\int_0^\tau \intO{ \lang \Nu;(r-\vr)( s (\vr, \vt)-s(r,\Theta)) \rang  \Big(\partial_t \Theta +\vU\cdot\Grad\Theta \Big)} \dt 
+\\
& +\int_0^\tau \intO{\Big(\lang \Nu;   \vr (s (\vr, \vt)-s(r,\Theta))(\vU-\vu) \rang\cdot \Grad \Theta  }\dt      -\\
	&+\int_0^\tau  \intO{ \lang \Nu;  \left(\frac{\vr}{r}-1\right)(\vU-\widetilde{ \vu})\rang\cdot\left( \Div\mathbb{S}(r,\Theta,\mathbb{D}(\vU))+\Curl \vH \times \vH \right) }\dt - \\
			& - \int_0^\tau \intO{ \lang \Nu; \vr (\vu-\vU) \otimes (\vu-\vU) \rang : \Grad \vU }\dt			+\\
				&+\int_{0}^{\tau} \intO{\lang \Nu;(\vU-\vu)\times(\vB-\vH)\cdot\Curl\vH\rang }\dt +\\
				&+ \int_0^\tau  \intO{    \lang \Nu; (\vB-\vH)\cdot\Grad\vU\cdot(\vB-\vH) - \frac{1}{2}\Div\vU|\vH-\vB|^2\rang}\dt .
\end{aligned}
\end{equation}
Further, by Gibb's equation, the entropy and the continuity equations satisfied by the strong solution we derive the identity
\begin{align*}
&\lang\Nu;\left(1-\frac{\vr}{r} \right)\rang\left(\partial_t p(r,\Theta) + \vU\cdot\Grad p(r,\Theta) \right)+\Div\vU \lang\Nu;p(r,\Theta)-p(\vr,\vt)\rang= &\\
&=\Div\vU \lang\Nu; \left( p(r,\Theta) - \frac{\partial p(r,\Theta)}{\partial \vr}(r-\vr)- \frac{\partial p(r,\Theta)}{\partial \vt}(\Theta-\vt)-p(\vr,\vt)\right)\rang-\\
&-\lang\Nu;r(r-\vr)\rang\frac{\partial s(r,\Theta)}{\partial\vr}\left(\partial_t\Theta+\vU\cdot\Grad\Theta \right)-\lang\Nu;r(\Theta-\vt)\rang\frac{\partial s(r,\Theta)}{\partial\vt}\left(\partial_t\Theta+\vU\cdot\Grad\Theta \right)-\\
&+\lang\Nu;\Theta-\vt\rang\Div\left(\frac{\kappa(r,\Theta)\Grad\Theta}{\Theta}\right)+\\
&+\lang\Nu;\left(1-\frac{\vt}{\Theta}\right)\rang\left(\mathbb{S}(r,\Theta,\mathbb{D}(\vU)):\Grad\vU+\frac{\kappa(r,\Theta)|\Grad\Theta|^2}{\Theta}+\zeta(r,\Theta)|\Curl\vH|^2 \right).
\end{align*}
Using \eqref{cc_S}, \eqref{cc_t} and rearranging the terms in a suitable way, the inequality \eqref{rei_1} becomes
 \begin{equation}
\begin{aligned}\label{rei_2}
&\Enu(\tau)+\mathcal{D}_{\Theta}(\tau) +\mathcal{D}_{\vH}(\tau) +\\
&+\int_0^\tau \intO{ \lang	\Nu;\frac{\Theta}{\vt} \mathbb{S}(\vr,\vt, \Du):\Du \rang +\lang \Nu; \frac{\vt}{\Theta}\mathbb{S}(r,\Theta,\mathbb{D}(\vU)):\Grad\vU\rang }\dt +\\
&+ \int_0^\tau \intO{ \Theta\lang	\Nu; \frac{\kappa(\vr,\vt)\vc{D}_{\vt}}{\vt}\cdot\left( \frac{\vc{D}_{\vt}}{\vt}-\frac{\Grad\Theta}{\Theta}\right)\rang+ \frac{\kappa(r,\Theta)\Grad\Theta}{\Theta}\cdot\lang\Nu;\vt\left(\frac{\Grad\Theta}{\Theta}-\frac{\mathbf{D}_\vt}{\vt}\right)\rang}\dt+&\\
&+ \int_0^\tau \intO{ \lang	\Nu;\frac{\zeta (\vr,\vt)\Theta}{\vt} |\vc{C}_{\vB} |^2  \rang +\lang\Nu;\frac{\zeta(r,\Theta)\vt}{\Theta}|\Curl\vH|^2\rang} \dt - \\
&- \int_0^\tau \intO{  \lang \Nu; \mathbb{S}(\vr,\vt, \Du)\rang: \Grad \vU  }\dt  - \int_0^\tau  \intO{  \lang \Nu;\Du\rang\cdot\mathbb{S}(r,\Theta,\mathbb{D}(\vU)) }\dt-\\
&-\int_0^\tau  \intO{   \zeta (r,\Theta) \Curl\vH \cdot \lang \Nu; \vc{C}_\vB\rang  }\dt-\int_0^\tau  \intO{   \lang \Nu;  \zeta (\vr,\vt) \vc{C}_{\vB} \rang \cdot \Curl \vH }\dt\\
&\leq  R_2,
\end{aligned}
\end{equation}
where $R_2$ reads
\begin{equation}
\begin{aligned}\label{R2}
R_2&=
\int_0^\tau \intO{\Div\vU \lang\Nu; \left( p(r,\Theta) - \frac{\partial p(r,\Theta)}{\partial \vr}(r-\vr)- \frac{\partial p(r,\Theta)}{\partial \vt}(\Theta-\vt)-p(\vr,\vt)\right)\rang}\dt-\\
& -\int_0^\tau \intO{r \lang \Nu;  (r-\vr) \frac{\partial s(r,\Theta)}{\partial\vr}+(\Theta-\vt)\frac{\partial s(r,\Theta)}{\partial\vt}\rang  \Big(\partial_t \Theta +\vU\cdot\Grad\Theta \Big)} \dt -\\
&- \int_0^\tau \intO{ \lang \Nu;\vr( s (\vr, \vt)-s(r,\Theta)) \rang  \Big(\partial_t \Theta +\vU\cdot\Grad\Theta \Big)} \dt 
+\\
& +\int_0^\tau \intO{\Big(\lang \Nu;   \vr (s (\vr, \vt)-s(r,\Theta))(\vU-\vu) \rang\cdot \Grad \Theta  }\dt      -\\
	&+\int_0^\tau  \intO{ \lang \Nu;  \left(\frac{\vr}{r}-1\right)(\vU-\widetilde{ \vu})\rang\cdot\left( \Div\mathbb{S}(r,\Theta,\mathbb{D}(\vU))+\Curl \vH \times \vH \right) }\dt - \\
			& - \int_0^\tau \intO{ \lang \Nu; \vr (\vu-\vU) \otimes (\vu-\vU) \rang : \Grad \vU }\dt			+\\
				&+\int_{0}^{\tau} \intO{\lang \Nu;(\vU-\vu)\times(\vB-\vH)\cdot\Curl\vH\rang }\dt +\\
				&+ \int_0^\tau  \intO{    \lang \Nu; (\vB-\vH)\cdot\Grad\vU\cdot(\vB-\vH) - \frac{1}{2}\Div\vU|\vH-\vB|^2\rang}\dt +\\
& + c\int_0^\tau \left(\mathcal{D}_{\Theta}(t)+\mathcal{D}_{\vH}(t)\right) \dt.
\end{aligned}
\end{equation}
Note that we have used \eqref{cd_m} and \eqref{cd_b} to get 
\begin{align*}
-	\int_0^\tau \int_{\Ov{\Omega}}{ \Grad\vU: {\rm d}\mathfrak{R}^M \dx } \dt  -	\int_0^\tau \int_{\Ov{\Omega}}{ \Grad\vU: {\rm d}\mathfrak{R}^B\dx} \dt  \leq  c\int_0^\tau \left(\mathcal{D}_{\Theta}(t)+\mathcal{D}_{\vH}(t)\right) \dt.
\end{align*}
We further rewrite the integrals on the left-hand side of \eqref{rei_2} using Newton's rheological law \eqref{S}
and straightforward calculations, 
\begin{align*}
&\Enu(\tau)+ \mathcal{D}_{\Theta}(\tau) +\mathcal{D}_{\vH}(\tau)+\int_0^\tau\intO{\lang\Nu; \frac{\mu(\vr,\vt)}{2}\left|\sqrt{\frac{\Theta}{\vt}}\mathbb{T}[\Du]:-\sqrt{\frac{\vt}{\Theta}}\mathbb{T}[\mathbb{D}(\vU)]\right|^2 \rang}\dt -\\
&- \frac 1 2\int_0^\tau\intO{\mathbb{T}[\mathbb{D}(\vU)]:\lang\Nu; \left(\mu(\vr,\vt)-\mu(r,\Theta)\right)\left(\frac{\vt}{\Theta}\mathbb{T}[\mathbb{D}(\vU)]-\mathbb{T}[\Du]\right)\rang}\dt+\\
&+\int_0^\tau\intO{\lang\Nu;\eta(\vr,\vt)\left|\sqrt{\frac{\Theta}{\vt}}\tr(\Du):-\sqrt{\frac{\vt}{\Theta}}\Div\vU \right|^2\rang}\dt-\\
&- \int_0^\tau\intO{\Div\vU\lang\Nu;\left(\eta(\vr,\vt)-\eta(r,\Theta) \right)\left(\frac{\vt}{\Theta}\Div\vU-\tr(\Du)\right)\rang}\dt +\\
&+ \int_0^\tau\intO{\Theta\lang\Nu;\kappa(\vr,\vt)\left|\frac{\mathbf{D}_\vt}{\vt}-\frac{\Grad\Theta}{\Theta}\right|^2\rang+\kappa(r,\Theta)\frac{\Grad\Theta}{\Theta}\cdot\lang\Nu;(\vt-\Theta)\left(\frac{\Grad\Theta}{\Theta}-\frac{\mathbf{D}_\vt}{\vt}\right)\rang}\dt-\\
&-\int_0^\tau\intO{\Grad\Theta\cdot\lang\Nu;\left(\kappa(\vr,\vt)-\kappa(r,\Theta)\right)\left(\frac{\Grad\Theta}{\Theta}-\frac{\mathbf{D}_\vt}{\vt}\right)\rang}\dt+\\
&+ \int_0^\tau \intO{ \lang	\Nu;\zeta (\vr,\vt)\left|\sqrt{\frac{\Theta}{\vt}}\vc{C}_{\vB}-\sqrt{\frac{\vt}{\Theta}}\Curl\vH \right|^2  \rang}\dt -\\
&-\int_0^\tau\intO{\Curl\vH\cdot\lang\Nu;(\zeta(\vr,\vt)-\zeta(r,\Theta))\left(\frac{\vt}{\Theta}\Curl\vH-\vc{C}_\vB\right)\rang} \dt 
\leq R_2
\end{align*}
with $R_2$ given in \eqref{R2}. 

\subsection{Suitable form of relative energy inequality}

Following the approach of \cite{FeiNoSL} we further introduce a cut-off function 
\begin{align*}
\psi_\delta \in C^\infty_c((0,\infty)^2), \quad 0\leq &\psi_\delta(\vr,\vt)\leq 1,\\ &\psi_\delta(\vr,\vt)=1 \quad \mbox{whenever} \quad \delta <\vr < \frac{1}{\delta} \mbox{ and } \delta < \vt<\frac{1}{\delta}  \mbox{ for some } \delta>0.
\end{align*} 
Any measurable function $h(\vr,\vu,\vt,\vB,\Du,\DT,\vc{C}_\vB)$ can be written as $h=[h]_{\rm ess}+[h]_{\rm res},$ where  
\begin{align*}
[h]_{\rm ess}&=\psi_\delta(\vr,\vt)h(\vr,\vu,\vt,\vB,\Du,\DT,\vc{C}_\vB),\quad
[h]_{\rm res}=(1-\psi_\delta(\vr,\vt))h(\vr,\vu,\vt,\vB,\Du,\DT,\vc{C}_\vB)
\end{align*}
are referred to as  the {\it essential} and {\it residual} component of a function $h,$ respectively.
Under the assumption on thermodynamic functions being related through Gibb's equation \eqref{gibbs} and satisfying the hypothesis of thermodynamic stability \eqref{hts} one can show that   
\begin{equation}
\begin{aligned}\label{re_lb}
 E(\vr,\vt,\vu,\vB \; | \; r,\Theta,\vU,\vH ) 
\geq  c(\delta)&\bigg\{
|[\vr-r]_{\rm ess}|^2 +|[\vt-\Theta]_{\rm ess}|^2+|[\vu-\vU]_{\rm ess}|^2+|[\vB-\vH]_{\rm ess}|^2+ \bigg.\\
&\bigg. + 1_{\rm res}+[\vr]_{\rm res}+[\vr|s(\vr,\vt)|]_{\rm res}+[\vr e(\vr,\vt)]_{\rm res}+[\vr|\vu|^2]_{\rm res}+[|\vB|^2]_{\rm res}
\bigg\}.
 \end{aligned}
 \end{equation}
The analogous bound for the compressible Navier-Stokes-Fourier system was shown in \cite[Proposition 3.2]{FeiNoSL}.

Let us denote the nine integrals of $R_2$ in \eqref{R2}
by $r_i,$ $i=1, \ldots,9.$ In what follows we estimate them step by step. 
For the first integral we get by  \eqref{re_lb} that
\begin{align*}
r_1 & \leq c\int_0^\tau \intO{\lang\Nu; \left| p(r,\Theta) - \frac{\partial p(r,\Theta)}{\partial \vr}(r-\vr)- \frac{\partial p(r,\Theta)}{\partial \vt}(\Theta-\vt)-p(\vr,\vt)\right|\rang}\dt \\
& \leq c\int_0^\tau \intO{\lang\Nu;|[\vr-r]_{\rm ess}^2|+|[\vt-\Theta]_{\rm ess}|^2 \rang}\dt + c\int_0^\tau \intO{\lang\Nu;[1+\vr+\vt+|p(\vr,\vt)|]_{\rm res}\rang}\dt\\
&\leq c \int_0^\tau \Enu(\tau)\dt + c\int_0^\tau \intO{\lang\Nu;[\vt+|p(\vr,\vt)|]_{\rm res}\rang}\dt.
\end{align*}
The estimate of the second and the third integral reads
\begin{align*}
r_2&+r_3
=-\int_0^\tau \intO{\lang\Nu;\vr[s(\vr,\vt)-s(r,\Theta)]_{\rm res}\rang\left( \partial_t\Theta+\vU\cdot\Grad\Theta\right)}\dt - \\
&- \int_0^\tau \intO{\lang\Nu; (\vr-r)[s(\vr,\vt)-s(r,\Theta)]_{\rm ess} \rang \left( \partial_t\Theta+\vU\cdot\Grad\Theta\right)}\dt-\\
&- \int_0^\tau \intO{\lang\Nu;  r\left[s(\vr,\vt)+(r-\vr)\frac{\partial s(r,\Theta)}{\partial \vr}+(\Theta-\vt)\frac{\partial s(r,\Theta)}{\partial \vt}-s(r,\Theta)\right]_{\rm ess}\rang \left( \partial_t\Theta+\vU\cdot\Grad\Theta\right)}\dt-\\
&-\int_0^\tau \intO{\lang\Nu;  r\left[(r-\vr)\frac{\partial s(r,\Theta)}{\partial \vr}+(\Theta-\vt)\frac{\partial s(r,\Theta)}{\partial \vt}\right]_{\rm res}\rang \left( \partial_t\Theta+\vU\cdot\Grad\Theta\right)}\dt \\
&\leq c\int_0^\tau \intO{\lang\Nu;|[\vr-r]_{\rm ess}|^2+|[\vt-\Theta]_{\rm ess}|^2\rang}\dt + c \int_0^\tau \intO{\lang\Nu;\left[1+\vr |s (\vr,\vt)|+\vr+\vt\right]_{\rm res}\rang}\dt\\
& \leq c\int_0^\tau \Enu(t)\dt+c\int_0^\tau\intO{\lang\Nu; [\vt]_{\rm res}\rang}\dt.
\end{align*}
For the fourth and the fifth integral we have
\begin{align*}
r_4 & \leq c\int_0^\tau \intO{\lang \Nu; \vr|s(\vr,\vt)-s(r,\Theta)| |\vu-\vU|\rang} \dt\\
&\leq  c\int_0^\tau \intO{\lang\Nu;|[\vr-r]_{\rm ess}|^2+|[\vt-\Theta]_{\rm ess}|^2+|[\vu-\vU]_{\rm ess}|^2\rang +\lang\Nu;[\vr+\vr|s(\vr,\vt)|]_{\rm res}|[\vu]_{\rm res}|\rang}\dt\\
&\leq \int_0^\tau \Enu(t)\dt +  c \int_0^\tau \intO{\lang\Nu;[\vr+\vr|s(\vr,\vt)|]_{\rm res}|[\vu]_{\rm res}|\rang}\dt,\\
r_5 &\leq c \int_{0}^{\tau} \intO{\lang \Nu; |\vr-r||\vu-\vU|\rang}\dt \\
&\leq  c\int_0^\tau \intO{\lang\Nu;|[\vr-r]_{\rm ess}|^2+|[\vu-\vU]_{\rm ess}|^2\rang +\lang\Nu;[1+\vr]_{\rm res}|[\vu-\vU]_{\rm res}|\rang}\dt\\
&\leq \int_0^\tau \Enu(t)\dt + c \int_0^\tau \intO{\lang\Nu;[1+\vr]_{\rm res}|[\vu-\vU]_{\rm res}|\rang}\dt.
\end{align*}
The sixth integral can be directly bounded as
\begin{align*}
r_6 & \leq c\int_0^\tau\intO{\lang\Nu;\frac{1}{2}\vr|\vu-\vU|^2\rang}\dt \leq   c\int_0^\tau \Enu(t)\dt.
\end{align*}
For the two integrals containing magnetic induction we get
\begin{align*}
r_7 &\leq c \int_{0}^{\tau} \intO{\lang \Nu;|\vU-\vu||\vB-\vH|\rang}\dt  \\
&\leq \intO{\lang \Nu;|[\vU-\vu]_{\rm ess}|^2 + |[\vB-\vH]_{\rm ess}|^2\rang + \lang \Nu;[1+|\vB|]_{\rm res}|[\vu-\vU]_{\rm res}|\rang}\dt\\
&\leq c\int_0^\tau \Enu(t) \dt+ \intO{\lang \Nu;[1+|\vB|]_{\rm res}|[\vu-\vU]_{\rm res}|\rang}\dt,\\
r_8 & \leq c\int_{0}^{\tau} \intO{\lang \Nu;\frac{1}{2}|\vB-\vH|^2\rang}\dt \leq  c\int_0^\tau \Enu(t) \dt.
\end{align*}
To conclude, let us denote 
\begin{align*}
\mathcal{H}(\tau)=\Enu(\tau)+\mathcal{D}_\Theta(\tau)+\mathcal{D}_\vH(\tau).
\end{align*}  
The relative energy inequality finally takes the form 
\begin{equation}
\begin{aligned}
&\mathcal{H}(\tau)+\int_0^\tau\intO{\lang\Nu; \frac{\mu(\vr,\vt)}{2}\left|\sqrt{\frac{\Theta}{\vt}}\mathbb{T}[\Du]-\sqrt{\frac{\vt}{\Theta}}\mathbb{T}[\mathbb{D}(\vU)]\right|^2 \rang}\dt -\\
&- \frac 1 2\int_0^\tau\intO{\mathbb{T}[\mathbb{D}(\vU)]:\lang\Nu; \left(\mu(\vr,\vt)-\mu(r,\Theta)\right)\left(\frac{\vt}{\Theta}\mathbb{T}[\mathbb{D}(\vU)]-\mathbb{T}[\Du]\right)\rang}\dt+\\
&+\int_0^\tau\intO{\lang\Nu;\eta(\vr,\vt)\left|\sqrt{\frac{\Theta}{\vt}}\tr[\Du]-\sqrt{\frac{\vt}{\Theta}}\Div\vU \right|^2\rang}\dt-\\
&- \int_0^\tau\intO{\Div\vU\lang\Nu;\left(\eta(\vr,\vt)-\eta(r,\Theta) \right)\left(\frac{\vt}{\Theta}\Div\vU-\tr[\Du]\right)\rang}\dt +\\
&+ \int_0^\tau\intO{\Theta\lang\Nu;\kappa(\vr,\vt)\left|\frac{\mathbf{D}_\vt}{\vt}-\frac{\Grad\Theta}{\Theta}\right|^2\rang+\kappa(r,\Theta)\frac{\Grad\Theta}{\Theta}\cdot\lang\Nu;(\vt-\Theta)\left(\frac{\Grad\Theta}{\Theta}-\frac{\mathbf{D}_\vt}{\vt}\right)\rang}\dt-\\
&-\int_0^\tau\intO{\Grad\Theta\cdot\lang\Nu;\left(\kappa(\vr,\vt)-\kappa(r,\Theta)\right)\left(\frac{\Grad\Theta}{\Theta}-\frac{\mathbf{D}_\vt}{\vt}\right)\rang}\dt+\\
&+ \int_0^\tau \intO{ \lang	\Nu;\zeta (\vr,\vt)\left|\sqrt{\frac{\Theta}{\vt}}\vc{C}_{\vB}-\sqrt{\frac{\vt}{\Theta}}\Curl\vH \right|^2  \rang}\dt -\\
&-\int_0^\tau\intO{\Curl\vH\cdot\lang\Nu;(\zeta(\vr,\vt)-\zeta(r,\Theta))\left(\frac{\vt}{\Theta}\Curl\vH-\vc{C}_\vB\right)\rang} \dt \\
&\leq c\int_0^\tau \mathcal{H}(t)\dt +\int_0^\tau\intO{\lang\Nu;\left[\vt+|p(\vr,\vt)|+|\vu-\vU|+\vr|s(\vr,\vt)||\vu|+|\vB||\vu-\vU|\right]_{\rm res}\rang}\dt.
\end{aligned}\label{rei_final}
\end{equation}
So far we have derived the relative energy inequality using Gibb's equation \eqref{gibbs} and assuming that $(r,\Theta,\vU,\vH)$ satisfying \eqref{smooth_cc} is a strong solution to the MHD system with the same initial and boundary data as dissipative measure-valued solution  from Definition~\ref{D:mvs}. It means we have just proven the statement of the  following lemma.

\begin{Lemma}[{\bf Relative energy inequality for MHD system}]\label{L:rei}
Let the coefficients $\mu,$ $\eta,$ $\kappa$ and $\zeta$ be continuously differentiable functions of  $(\vr,\vt)\in(0,\infty)^2.$   Let the thermodynamic functions $p$, $e$ and $s$ be continuously differentiable functions of $(\vr,\vt)\in(0,\infty)^2$ satisfying Gibb's equation \eqref{gibbs} and the hypothesis of thermodynamic stability \eqref{hts}. Let $(r,\Theta,\vU,\vH)$ be a sufficiently regular strong solution to the MHD system \eqref{mhd_ce}-\eqref{mhd_ie}, \eqref{mhd_eb} with \eqref{S}, \eqref{F}, with the initial state \eqref{ic}, and satisfying boundary conditions \eqref{bc_v_D},  \eqref{bc_t_D}, \eqref{bc_t_N}, \eqref{bc_m_D} and \eqref{bc_m_N}. Let $\mathcal{V}$ be a dissipative measure-valued solution of the same problem according to Definition~\ref{D:mvs}.

Then the relative energy inequality \eqref{rei_final} holds true.
\end{Lemma}

\section{DMV-strong uniqueness principle}\label{S-UP}

This section contains the main result: (dissipative measure-valued)-strong uniqueness principles for the DMV solution of the full compressible  magnetohydrodynamic system. 
Our aim is to apply the Gronwall lemma to inequality \eqref{rei_final} to conclude that the two solutions coincide. It is enough to show that, under certain assumptions, all terms on the right-hand side of \eqref{rei_final} can be either absorb by the non-negative terms on the left-hand side or controlled by the relative energy.  

\subsection{Conditional DMV-strong uniqueness}

Let us begin with the {\it conditional} results, meaning that additional assumptions on the dissipative measure-valued solution have to be imposed in order to show that the DMV and  strong solutions coincide on the lifespan of the latter when emanating from the same initial data.

\subsubsection{Assuming bounded measure-valued solution}

\begin{Theorem}[{\bf Conditional result: bounded DMV}]\label{T:c1}
Let the coefficients $\mu,$ $\eta,$ $\kappa$ and $\zeta$ be continuously differentiable functions of $(\vr,\vt)\in(0,\infty)^2$ such that 
$
\mu>0, \ \eta \geq 0, \ \kappa>0, \ \zeta>0.
$
Let the thermodynamic functions $p$, $e$ and $s$ be continuously differentiable functions of $(\vr,\vt)\in(0,\infty)^2$ satisfying Gibb's equation \eqref{gibbs} and the hypothesis of thermodynamic stability \eqref{hts}. Let $(r,\Theta,\vU,\vH)$ be a strong solution to the MHD system satisfying the regularity assumptions \eqref{smooth_cc} and emanating from the initial data \eqref{ic}. Let $\mathcal{V}$ be a DMV solution of the same problem according to Definition~\ref{D:mvs}. Let us assume that 
\begin{align}\label{ass_bmvs}
\Nu\left\{ 0< \underline{\vr} < \vr < \overline{\vr}, \  0< \underline{\vt} < \vt < \overline{\vt}\right\}=1 \ \mbox{ for a.e. } (t,x)\in (0,T)\times\Omega 
\end{align}
for some constants $\underline{\vr},$ $\overline{\vr},$ $\underline{\vt},$ $\overline{\vt} >0,$ 
and the initial data coincide, i.e., 
\begin{align*}
\mathcal{V}_{0,x}=\delta_{[\vr_0,\vu_0,\vt_0,\vB_0]} \ \mbox{ for a.e. } x\in \Omega.
\end{align*} 

Then 
\begin{align*}
\Nu=\delta_{[r(t,x),\vU(t,x),\Theta(t,x),\vH(t,x),\mathbb{D}(\vu)(t,x),\Grad\Theta(t,x),\Curl\vH(t,x)]} \ \mbox{ for a.e. } (t,x)\in (0,T)\times\Omega.
\end{align*}
\end{Theorem}

\begin{proof}
Assuming \eqref{smooth_cc} and \eqref{ass_bmvs}  inequality \eqref{re_lb} yields 
\begin{align}\label{cr_aux1}
\Enu(\tau) \geq
c \intO{\lang  \mathcal{V}_{\tau,x}; 
|\vr-r|^2+|\vt-\Theta|^2+|\vu-\vU|^2+|\vB-\vH|^2\rang}\dt,
\end{align}
see \cite[Proposition 1]{FeiDCDS} and \cite[Proposition 3.2]{FeiNoSL} for more details. For the integrals on the left-hand side of \eqref{rei_final} we now get 
\begin{equation}
\begin{aligned}\label{cr_aux2}
&\int_0^\tau\intO{\lang\Nu; \frac{\mu(\vr,\vt)}{2}\left|\sqrt{\frac{\Theta}{\vt}}\mathbb{T}[\Du]-\sqrt{\frac{\vt}{\Theta}}\mathbb{T}[\mathbb{D}(\vU)]\right|^2 \rang +\lang\Nu;\eta(\vr,\vt)\left|\sqrt{\frac{\Theta}{\vt}}\tr[\Du]-\sqrt{\frac{\vt}{\Theta}}\Div\vU \right|^2\rang}\dt+\\
&+ \int_0^\tau\intO{\Theta\lang\Nu;\kappa(\vr,\vt)\left|\frac{\mathbf{D}_\vt}{\vt}-\frac{\Grad\Theta}{\Theta}\right|^2\rang}\dt+ \int_0^\tau \intO{ \lang	\Nu;\zeta (\vr,\vt)\left|\sqrt{\frac{\Theta}{\vt}}\vc{C}_{\vB}-\sqrt{\frac{\vt}{\Theta}}\Curl\vH \right|^2  \rang}\dt \\
&\geq c\int_0^\tau\intO{\lang\Nu;\left|\sqrt{\frac{\Theta}{\vt}}\mathbb{T}[\Du]-\sqrt{\frac{\vt}{\Theta}}\mathbb{T}[\mathbb{D}(\vU)]\right|^2\rang+\lang\Nu; \left|\sqrt{\frac{\Theta}{\vt}}\tr[\Du]-\sqrt{\frac{\vt}{\Theta}}\Div\vU \right|^2\rang}\dt+\\
&+ c\int_0^\tau\intO{\lang\Nu;\left|\frac{\mathbf{D}_\vt}{\vt}-\frac{\Grad\Theta}{\Theta}\right|^2\rang+\lang\Nu;\left|\sqrt{\frac{\Theta}{\vt}}\vc{C}_{\vB}-\sqrt{\frac{\vt}{\Theta}}\Curl\vH \right|^2 \rang}\dt,
\end{aligned}
\end{equation}
and 
\begin{equation}
\begin{aligned}\label{cr_aux3}
&- \frac 1 2\int_0^\tau\intO{\mathbb{T}[\mathbb{D}(\vU)]:\lang\Nu; \left(\mu(\vr,\vt)-\mu(r,\Theta)\right)\left(\frac{\vt}{\Theta}\mathbb{T}[\mathbb{D}(\vU)]-\mathbb{T}[\Du]\right)\rang}\dt+\\
&- \int_0^\tau\intO{\Div\vU\lang\Nu;\left(\eta(\vr,\vt)-\eta(r,\Theta) \right)\left(\frac{\vt}{\Theta}\Div\vU-\tr[\Du]\right)\rang}\dt +\\
&+\int_0^\tau\intO{\kappa(r,\Theta)\frac{\Grad\Theta}{\Theta}\cdot\lang\Nu;(\vt-\Theta)\left(\frac{\Grad\Theta}{\Theta}-\frac{\mathbf{D}_\vt}{\vt}\right)\rang}\dt-\\
&-\int_0^\tau\intO{\Grad\Theta\cdot\lang\Nu;\left(\kappa(\vr,\vt)-\kappa(r,\Theta)\right)\left(\frac{\Grad\Theta}{\Theta}-\frac{\mathbf{D}_\vt}{\vt}\right)\rang}\dt-\\
&-\int_0^\tau\intO{\Curl\vH\cdot\lang\Nu;(\zeta(\vr,\vt)-\zeta(r,\Theta))\left(\frac{\Theta}{\vt}\Curl\vH-\vc{C}_\vB\right)\rang} \dt \\
&\leq c \int_0^\tau\intO{\lang\Nu;|\vr-r|^2+|\vt-\Theta|^2\rang}\dt+\\
&+\varepsilon\int_0^\tau\intO{\lang\Nu;\left|\sqrt{\frac{\Theta}{\vt}}\mathbb{T}[\Du]-\sqrt{\frac{\vt}{\Theta}}\mathbb{T}[\mathbb{D}(\vU)]\right|^2\rang+\lang\Nu; \left|\sqrt{\frac{\Theta}{\vt}}\tr[\Du]-\sqrt{\frac{\vt}{\Theta}}\Div\vU \right|^2\rang}\dt+\\
&+ \varepsilon\int_0^\tau\intO{\lang\Nu;\left|\frac{\mathbf{D}_\vt}{\vt}-\frac{\Grad\Theta}{\Theta}\right|^2\rang+\lang\Nu;\left|\sqrt{\frac{\Theta}{\vt}}\vc{C}_{\vB}-\sqrt{\frac{\vt}{\Theta}}\Curl\vH \right|^2 \rang}\dt,
\end{aligned}
\end{equation}
for $\varepsilon>0$ small enough such that the latter integrals can be absorbed by the left-hand side. Under the assumption    \eqref{ass_bmvs} we have 
\begin{align}\label{cr_aux4}
\int_0^\tau\intO{\lang\Nu;\big[\vt+|p(\vr,\vt)|+|\vu-\vU|+\vr|s(\vr,\vt)||\vu|+|\vB||\vu-\vU|\big]_{\rm res}\rang}\dt=0.
\end{align}
Combining \eqref{rei_final} with \eqref{cr_aux1} - \eqref{cr_aux4} we get 
\begin{align*}
\mathcal{H}(\tau)&+c\int_0^\tau\intO{\lang\Nu;\left|\sqrt{\frac{\Theta}{\vt}}\mathbb{T}[\Du]-\sqrt{\frac{\vt}{\Theta}}\mathbb{T}[\mathbb{D}(\vU)]\right|^2\rang+\lang\Nu; \left|\sqrt{\frac{\Theta}{\vt}}\tr[\Du]-\sqrt{\frac{\vt}{\Theta}}\Div\vU \right|^2\rang}\dt+\\
&+ c\int_0^\tau\intO{\lang\Nu;\left|\frac{\mathbf{D}_\vt}{\vt}-\frac{\Grad\Theta}{\Theta}\right|^2\rang+\lang\Nu;\left|\sqrt{\frac{\Theta}{\vt}}\vc{C}_{\vB}-\sqrt{\frac{\vt}{\Theta}}\Curl\vH \right|^2 \rang}\dt \leq c\int_0^\tau\mathcal{H}(t)\dt,
\end{align*}
where  $c$ is a constant depending on the norms of strong solution and constants $\underline{\vr},$ $\overline{\vr},$ $\underline{\vt},$ $\overline{\vt} >0.$ 
The Gronwall lemma  yields the desired result. 

  \end{proof}

\subsubsection{Assuming constant coefficients} 
 
From now on we impose a structural condition on pressure, cf. \cite{BreFeiNo}. Note that it is not a   restrictive hypothesis, at least for gases for which $p\approx \vr e.$
We proceed further with another conditional result.
 
 \begin{Theorem}[{\bf Conditional result: constant coefficients}]\label{T:c2}
Let the coefficients $\mu,$ $\eta,$  $\kappa$ and $\zeta$ be positive constants.
Let the thermodynamic functions $p$, $e$ and $s$ be continuously differentiable functions of $(\vr,\vt)\in(0,\infty)^2$ satisfying Gibb's equation \eqref{gibbs} and the hypothesis of thermodynamic stability \eqref{hts}
with 
\begin{align}\label{ass_cc_p}
|p(\vr,\vt)|\aleq \big(1+\vr e(\vr,\vt)+\vr|s(\vr,\vt)|\big).
\end{align}
 Let $(r,\Theta,\vU,\vH)$ be a strong solution to the MHD system satisfying the regularity assumptions \eqref{smooth_cc} and emanating from the initial data  \eqref{ic}. Let $\mathcal{V}$ be a DMV solution of the same problem according to Definition~\ref{D:mvs}. Let us assume that 
\begin{align}\label{ass_cc}
\Nu\left\{\vt\leq \Ov{\vt}, |\vu|\leq\Ov{\vu}\right\}=1 \ \mbox{ for a.e. } (t,x)\in (0,T)\times\Omega 
\end{align}
for some constants  $\Ov{\vt},$ $\Ov{\vu},$  and the initial data coincide, i.e., 
\begin{align*}
\mathcal{V}_{0,x}=\delta_{[\vr_0,\vu_0,\vt_0,\vB_0]} \ \mbox{ for a.e. } x\in \Omega.
\end{align*} 

Then 
\begin{align*}
\Nu=\delta_{[r(t,x),\vU(t,x),\Theta(t,x),\vH(t,x),\mathbb{D}(\vu)(t,x),\Grad\Theta(t,x),\Curl\vH(t,x)]} \ \mbox{ for a.e. } (t,x)\in (0,T)\times\Omega.
\end{align*}
\end{Theorem}

\begin{proof}
We shall estimate  all integrals in \eqref{rei_final}  taking into account  the above  assumptions. The left-hand side directly equals
\begin{equation}
\begin{aligned}\label{cr_cc_aux2}
&\int_0^\tau\intO{\lang\Nu; \frac{\mu}{2}\left|\sqrt{\frac{\Theta}{\vt}}\mathbb{T}[\Du]-\sqrt{\frac{\vt}{\Theta}}\mathbb{T}[\mathbb{D}(\vU)]\right|^2 \rang +\lang\Nu;\eta\left|\sqrt{\frac{\Theta}{\vt}}\tr[\Du]-\sqrt{\frac{\vt}{\Theta}}\Div\vU \right|^2\rang}\dt+\\
&+ \int_0^\tau\intO{\Theta\lang\Nu;\kappa\left|\frac{\mathbf{D}_\vt}{\vt}-\frac{\Grad\Theta}{\Theta}\right|^2\rang}\dt+ \int_0^\tau \intO{ \lang	\Nu;\zeta \left|\sqrt{\frac{\Theta}{\vt}}\vc{C}_{\vB}-\sqrt{\frac{\vt}{\Theta}}\Curl\vH \right|^2  \rang}\dt +\\
&+\int_0^\tau\intO{\kappa\frac{\Grad\Theta}{\Theta}\cdot\lang\Nu;(\vt-\Theta)\left(\frac{\Grad\Theta}{\Theta}-\frac{\mathbf{D}_\vt}{\vt}\right)\rang}\dt.
\end{aligned}
\end{equation}
Moreover, 
\begin{align*}
&\int_0^\tau\intO{\kappa\frac{\Grad\Theta}{\Theta}\cdot\lang\Nu;(\vt-\Theta)\left(\frac{\Grad\Theta}{\Theta}-\frac{\mathbf{D}_\vt}{\vt}\right)\rang}\dt\\
& \leq \int_0^\tau\intO{\inf\Theta\frac{\kappa}{2}\lang\Nu;\left|\frac{\Grad\Theta}{\Theta}-\frac{\mathbf{D}_\vt}{\vt}\right|^2\rang}\dt +c_\kappa \int_0^\tau\intO{\lang\Nu;|\vt-\Theta|^2\rang}\dt .
\end{align*}
Note that 
\begin{equation}
\begin{aligned}\label{temp}
&\int_0^\tau\intO{\lang\Nu;|\vt-\Theta|^2\rang}\dt \leq  \int_0^\tau\intO{\lang\Nu;|[\vt-\Theta]_{\rm ess}|^2 +|[\vt-\Theta]_{\rm res}|^2  \rang}\dt\\
& \leq c\int_0^\tau \Enu(t)\dt + c\intO{\lang\Nu;[1+\vt^2]_{\rm res}  \rang}\dt \leq c\int_0^\tau \Enu(t)\dt + c \intO{\lang\Nu;[\vt^2]_{\rm res}  \rang}\dt,
\end{aligned}
\end{equation}
where the last integral can be controlled by \eqref{ass_cc}.
In view of \eqref{ass_cc_p}, \eqref{ass_cc} and \eqref{re_lb}  the ``residual'' integral on the right-hand side of \eqref{rei_final} can be bounded by the relative energy. Indeed,
\begin{align*}
&\int_0^\tau\intO{\lang\Nu;\vt+|p(\vr,\vt)|+|\vu-\vU|+\vr|s(\vr,\vt)||\vu|+|\vB||\vu-\vU|]_{\rm res}\rang}\dt\\
& \leq  c\int_0^\tau\intO{\lang\Nu;\left[1+\vr e(\vr,\vt)+\vr|s(\vr,\vt)|  + |\vB|^2\right]_{\rm res}\rang }\dt \leq c \int_0^\tau\mathcal{H}(t)\dt.
\end{align*}
Hence, the Gronwall lemma can be applied and the result follows. 
\end{proof}

\subsubsection{Assuming perfect gas law}

To show the third conditional  DMV-strong uniqueness result we consider the case of Boyle's law for pressure which comes along with a certain condition on the entropy and  coefficients $\mu,$ $\eta,$ $\kappa,$ and $\zeta$. 

 \begin{Theorem}[{\bf Conditional result: perfect gas law}]\label{T:c3}
Let the coefficients $\mu,$ $\eta,$  $\kappa$ and $\zeta$ be  such that 
\begin{equation}
\begin{aligned}
&\mu(\vt)=\mu_0+\mu_1\vt, \ \mu_0>0,\ \mu_1> 0, \ &\eta(\vt)=\eta_0+\eta_1\vt, \ \eta_0\geq 0,\ \eta_1\geq 0, \\ & \kappa >0,  \ &\zeta(\vt)=\zeta_0+\zeta_1\vt, \ \zeta_0 >0,\ \zeta_1 \geq 0.\label{boyle_coeff}
\end{aligned}
\end{equation}
Let the thermodynamic functions $p$, $e$ and $s$ be continuously differentiable functions of $(\vr,\vt)\in(0,\infty)^2$ satisfying Gibb's equation \eqref{gibbs} and the hypothesis of thermodynamic stability \eqref{hts}
with 
\begin{align}\label{ass_pg}
p(\vr,\vt)=\vr\vt, \ e(\vr,\vt)=c_v\vt,\ s(\vr,\vt)=\log\left(\frac{\vt^{c_v}}{\vr}\right), \ c_v>1.
\end{align}
 Let $(r,\Theta,\vU,\vH)$ be a strong solution to the MHD system satisfying the regularity assumptions \eqref{smooth_cc} and emanating from the initial data  \eqref{ic}. Let $\mathcal{V}$ be a DMV solution of the same problem according to Definition~\ref{D:mvs}. Let us assume that 
\begin{align}\label{ass_pg_s}
\Nu\left\{|s(\vr,\vt)|\leq \Ov{s}\right\}=1 \ \mbox{ for a.e. } (t,x)\in (0,T)\times\Omega 
\end{align}
for a constant  $\Ov{s},$ 
and the initial data coincide, i.e., 
\begin{align*}
\mathcal{V}_{0,x}=\delta_{[\vr_0,\vu_0,\vt_0,\vB_0]} \ \mbox{ for a.e. } x\in \Omega.
\end{align*} 

Then 
\begin{align*}
\Nu=\delta_{[r(t,x),\vU(t,x),\Theta(t,x),\vH(t,x),\mathbb{D}(\vu)(t,x),\Grad\Theta(t,x),\Curl\vH(t,x)]} \ \mbox{ for a.e. } (t,x)\in (0,T)\times\Omega.
\end{align*}
\end{Theorem}

 \begin{proof}
 In accordance with \eqref{boyle_coeff} and \eqref{re_lb} the left-hand side of \eqref{rei_final} can be bounded from below by 
\begin{equation}
\begin{aligned}\label{cr_aux2}
&  \int_0^\tau\intO{\lang\Nu;\frac{\mu_0}{2}\left|\sqrt{\frac{\Theta}{\vt}}\mathbb{T}[\Du]-\sqrt{\frac{\vt}{\Theta}}\mathbb{T}[\mathbb{D}(\vU)]\right|^2+\eta_0\left|\sqrt{\frac{\Theta}{\vt}}\tr[\Du]-\sqrt{\frac{\vt}{\Theta}}\Div\vU \right|^2\rang}\dt+\\
&+ \inf\Theta\int_0^\tau\intO{\kappa\lang\Nu;\left|\frac{\mathbf{D}_\vt}{\vt}-\frac{\Grad\Theta}{\Theta}\right|^2\rang+\zeta_0\lang\Nu;\left|\sqrt{\frac{\Theta}{\vt}}\vc{C}_{\vB}-\sqrt{\frac{\vt}{\Theta}}\Curl\vH \right|^2 \rang}\dt+\\
&+\kappa\int_0^\tau\intO{\frac{\Grad\Theta}{\Theta}\cdot\lang\Nu;(\vt-\Theta)\left(\frac{\Grad\Theta}{\Theta}-\frac{\mathbf{D}_\vt}{\vt}\right)\rang}\dt+\\
&+\inf\Theta\int_0^\tau\intO{\frac{\mu_1}{2}\lang\Nu;\left|\mathbb{T}[\Du]-\mathbb{T}[\Grad\vU]\right|^2\rang +\eta_1\lang\Nu;\left|\tr[\Du]-\Div\vU \right|^2\rang }\dt+\\
&+\inf\Theta\int_0^\tau\intO{\zeta_1\lang\Nu;\left|\vc{C}_\vB-\Curl\vH\right|^2\rang}\dt-\\
&- \frac{\mu_1}{2}\int_0^\tau\intO{\mathbb{T}[\mathbb{D}(\vU)]:\lang\Nu; \left(\vt-\Theta\right)\left(\mathbb{T}[\Du]-\mathbb{T}[\mathbb{D}(\vU)]\right)\rang}\dt-\\
&- \eta_1\int_0^\tau\intO{\Div\vU\lang\Nu;\left(\vt-\Theta \right)\left(\tr[\Du]-\Div\vU\right)\rang}\dt -\\
&- \zeta_1\int_0^\tau\intO{\Curl\vH\cdot\lang\Nu; \left(\vt-\Theta\right)\left(\vc{C}_\vB-\Curl\vH\right)\rang}\dt.
\end{aligned}
\end{equation}
Moreover, 
\begin{equation}
\begin{aligned}\label{cr_aux3}
\kappa&\int_0^\tau\intO{\frac{\Grad\Theta}{\Theta}\cdot\lang\Nu;(\vt-\Theta)\left(\frac{\Grad\Theta}{\Theta}-\frac{\mathbf{D}_\vt}{\vt}\right)\rang}\dt \\
&\leq \varepsilon\int_0^\tau\intO{\lang\Nu;\left|\frac{\Grad\Theta}{\Theta}-\frac{\mathbf{D}_\vt}{\vt}\right|^2\rang}\dt+c_\varepsilon\int_0^\tau\intO{\lang\Nu;|\vt-\Theta|^2\rang}\dt, \\
\frac{\mu_1}{2}&\int_0^\tau\intO{\mathbb{T}[\mathbb{D}(\vU)]:\lang\Nu; \left(\vt-\Theta\right)\left(\mathbb{T}[\Du]-\mathbb{T}[\mathbb{D}(\vU)]\right)\rang}\dt \\
&\leq \varepsilon\int_0^\tau\intO{\lang\Nu;\left|\mathbb{T}[\Du]-\mathbb{T}[\mathbb{D}(\vU)]\right|^2\rang}\dt + c_\varepsilon \int_0^\tau\intO{\lang\Nu;|\vt-\Theta|^2\rang}\dt, \\
\eta_1&\int_0^\tau\intO{\Div\vU\lang\Nu;\left(\vt-\Theta \right)\left(\tr[\Du]-\Div\vU\right)\rang}\dt \\
&\leq  \varepsilon\int_0^\tau\intO{\lang\Nu; \left|\tr[\Du]-\Div\vU \right|^2\rang}\dt+ c_\varepsilon \int_0^\tau\intO{\lang\Nu;|\vt-\Theta|^2\rang}\dt, \\
\zeta_1&\int_0^\tau\intO{\Curl\vH\cdot\lang\Nu; \left(\vt-\Theta\right)\left(\vc{C}_\vB-\Curl\vH\right)\rang}\dt\\
&\leq \varepsilon\int_0^\tau\intO{\lang\Nu; \left|\vc{C}_\vB-\Curl\vH\right|^2\rang}\dt+ c_\varepsilon \int_0^\tau\intO{\lang\Nu;|\vt-\Theta|^2\rang}\dt.
\end{aligned}
\end{equation}
Recall \eqref{temp}. 
To control the ``residual'' integrals we begin with the observation that follows from \eqref{ass_pg}:
\begin{align*}
\vt^{c_v} \aleq \vr \Rightarrow \vt^{1+c_v} \aleq \vr\vt =\frac{1}{c_v}\vr e(\vr,\vt) .
\end{align*}
Thus, 
\begin{align*}
\intO{\lang\Nu;[\vt+\vt^2]_{\rm res}  \rang}\dt \leq \intO{\lang\Nu;[1+\vr e(\vr,\vt)]_{\rm res}  \rang}\dt .
\end{align*}
Further, by \eqref{ass_pg} and \eqref{ass_pg_s}, we have
\begin{align*}
\int_0^\tau\intO{\lang\Nu;\left[|p(\vr,\vt)|+\vr|s(\vr,\vt)| |\vu| \right]_{\rm res}\rang }\dt  \leq \int_0^\tau\intO{\lang\Nu;\left[\vr e(\vr,\vt)+\vr+\vr |\vu|^2 \right]_{\rm res}\rang }\dt.
\end{align*}
The Korn-Poincar\' e inequality \eqref{KP} yields 
\begin{align*}
&\int_0^\tau\intO{\lang\Nu;\left[|\vu-\vU|\right]_{\rm res}\rang }\dt \leq \int_0^\tau\intO{\lang\Nu;\left[\frac{1}{2\varepsilon}\right]_{\rm res}+\varepsilon |\vu-\vU|^2\rang }\dt \\
&\leq c\int_0^\tau\Enu(t)\dt+\varepsilon\int_0^\tau\intO{\lang\Nu;\left|\mathbb{T}[\Du]-\mathbb{T}[\mathbb{D}(\vU)]\right|^2\rang }\dt,
\end{align*}
 and analogously,
\begin{align*}
&\int_0^\tau\intO{\lang\Nu;\left[|\vB||\vu-\vU|\right]_{\rm res}\rang }\dt \leq \int_0^\tau\intO{\lang\Nu;\left[\frac{1}{2\varepsilon}|\vB|^2\right]_{\rm res}+\varepsilon |\vu-\vU|^2\rang }\dt \\
&\leq c_\varepsilon\int_0^\tau\Enu(t)\dt+ \varepsilon\int_0^\tau\intO{\lang\Nu;\left|\mathbb{T}[\Du]-\mathbb{T}[\mathbb{D}(\vU)]\right|^2\rang }\dt .
\end{align*}
Combining the above estimates we conclude
\begin{align*}
&\int_0^\tau\intO{\lang\Nu;\left[\vt^2+\vt+|p(\vr,\vt)|+|\vu-\vU|+\vr|s(\vr,\vt)| |\vu-\vU| +|\vB||\vu-\vU|\right]_{\rm res}\rang }\dt \\
&\leq c_\varepsilon\int_0^\tau\Enu(t)\dt+2 \varepsilon\int_0^\tau\intO{\lang\Nu;\left|\mathbb{T}[\Du]-\mathbb{T}[\mathbb{D}(\vU)]\right|^2\rang }\dt 
\end{align*}
for $\varepsilon>0$ small enough such that the last integral can be absorb by the same one on the left-hand side. 
The relative energy inequality finally reads
\begin{equation}
\begin{aligned}\label{cr_final_re}
 &\mathcal{H}(\tau)+\int_0^\tau\intO{\lang\Nu;\left|\sqrt{\frac{\Theta}{\vt}}\mathbb{T}[\Du]-\sqrt{\frac{\vt}{\Theta}}\mathbb{T}[\mathbb{D}(\vU)]\right|^2+\left|\sqrt{\frac{\Theta}{\vt}}\tr[\Du]-\sqrt{\frac{\vt}{\Theta}}\Div\vU \right|^2\rang}\dt+\\
&+ \int_0^\tau\intO{\lang\Nu;\left|\frac{\mathbf{D}_\vt}{\vt}-\frac{\Grad\Theta}{\Theta}\right|^2\rang+\lang\Nu;\left|\sqrt{\frac{\Theta}{\vt}}\vc{C}_{\vB}-\sqrt{\frac{\vt}{\Theta}}\Curl\vH \right|^2 \rang}\dt+\\
&+\int_0^\tau\intO{\lang\Nu;\left|\mathbb{T}[\Du]-\mathbb{T}[\mathbb{D}(\vU)]\right|^2\rang}\dt+\int_0^\tau\intO{\lang\Nu;\left|\tr[\Du]-\Div\vU \right|^2\rang}\dt+\\
&+\int_0^\tau\intO{\lang\Nu;\left|\vc{C}_\vB-\Curl\vH\right|^2\rang}\dt \aleq \int_0^\tau \mathcal{H}(t)\dt.
\end{aligned}
\end{equation}
Obviously, the Gronwall lemma yields the desired result. 
 \end{proof}
 
 \subsection{Unconditional DMV-strong uniqueness}

Finally, we present an {\it unconditional} (dissipative measure-valued)-strong uniqueness principle in the sense that no additional assumptions on DMV solution are imposed. 
The prize we need to pay are the equations of state based on the Third law of thermodynamic enforced through Gibb's equation as considered in, e.g. \cite{BreFeiNo,FeiNoSL}.

\subsubsection{Equation of state}

We suppose the pressure obeys the \emph{thermal equation of state} 
\begin{equation} \label{eosT}
	p(\vr, \vt)  = p_M(\vr, \vt) + p_R(\vt),
\end{equation}
while the internal energy satisfies the \emph{caloric equation of state}
\begin{equation} \label{eosC}
	e(\vr, \vt) = e_M (\vr, \vt) + e_R (\vr ,\vt).
\end{equation}	 
The  components $p_M$ and $e_M$ satisfy the relation characteristic for monoatomic gases,
\begin{equation} \label{eosM}
	p_M (\vr, \vt) = \frac{2}{3} \vr e_M(\vr, \vt). 
\end{equation} 
Now, it follows from Gibbs' relation \eqref{gibbs} applied to $p_M$, $e_M$ that they must take a general form
\begin{align}
	p_M(\vr,\vt) = \vt^{\frac{5}{2}} P \left( \frac{\vr}{\vt^{\frac{3}{2}}  } \right),\ e_M(\vr,\vt) =  \frac{3}{2} \frac{\vt^{\frac{5}{2}} }{\vr} P \left( \frac{\vr}{\vt^{\frac{3}{2}}  } \right)
	\label{w1}	
\end{align}
for a function $P \in C^1([0,\infty))$. In addition, the hypothesis of thermodynamics stability \eqref{hts}
applied to $p_M$, $e_M$ yields
\begin{equation} \label{w2}
	P(0) = 0,\ P'(Z) > 0,\ \frac{ \frac{5}{3} P(Z) - P'(Z) Z }{Z} > 0 \ \mbox{for all}\ Z \geq 0.
\end{equation} 	
In accordance with \eqref{w1}, the entropy takes the form
\begin{equation} \label{entropy}
	s(\vr, \vt) = s_M(\vr, \vt) + s_R (\vr, \vt),\ s_M(\vr, \vt) = \mathcal{S} \left( \frac{\vr}{\vt^{\frac{3}{2}} } \right),
\end{equation}
where
\begin{equation} \label{w6}
	\mathcal{S}'(Z) = -\frac{3}{2} \frac{ \frac{5}{3} P(Z) - P'(Z) Z }{Z^2}.
\end{equation}

Finally, we impose two technical but physically grounded hypotheses
in the degenerate area $\frac{\vr}{\vt^{\frac{3}{2}}} = Z >> 1.$
Firstly, it follows from 
\eqref{w2} that the function $Z \mapsto P(Z)/ Z^{\frac{5}{3}}$ is decreasing, and we suppose
\begin{equation} \label{w3}
	\lim_{Z \to \infty} \frac{ P(Z) }{Z^{\frac{5}{3}}} = p_\infty > 0.
\end{equation}
Secondly, by the same token, the function $Z \mapsto \mathcal{S}(Z)$ is decreasing, and we suppose 
\begin{equation} \label{w7}
	\lim_{Z \to \infty} \mathcal{S}(Z) = 0.	
\end{equation}	
Hypothesis \eqref{w3} reflects the effect of the electron pressure in the degenerate area, while 
\eqref{w7} is nothing other than the Third law of thermodynamics.

\begin{Theorem}[{\bf Unconditional result}]\label{T:ur}
Let the coefficients $\mu,$ $\eta,$  $\kappa$ and $\zeta$ be  such that 
\begin{equation}
\begin{aligned}
&\mu(\vt)=\mu_0+\mu_1\vt, \ \mu_0>0,\ \mu_1> 0, \ &\eta(\vt)=\eta_0+\eta_1\vt, \ \eta_0\geq 0,\ \eta_1\geq 0, \\ &\kappa >0,  \ &\zeta(\vt)=\zeta_0+\zeta_1\vt, \ \zeta_0 >0,\ \zeta_1 \geq 0.\label{ur_coeff}
\end{aligned}
\end{equation}
Let the thermodynamic functions $p$, $e$ and $s$ be given  by \eqref{eosT} - \eqref{w7} with 
\begin{align}
  p_R(\vt) = a \vt^2,\ e_R (\vr ,\vt) = \frac{a}{\vr} \vt^2, \ s_R(\vr, \vt) =  \frac{2a}{\vr}\vt, \ a > 0.
\end{align}
  Let $(r,\Theta,\vU,\vH)$ be a strong solution to the MHD system satisfying the regularity assumptions \eqref{smooth_cc} and emanating from the initial data  \eqref{ic}. Let $\mathcal{V}$ be a DMV solution of the same problem according to Definition~\ref{D:mvs} such that the initial data coincide, 
\begin{align}\label{pR}
\mathcal{V}_{0,x}=\delta_{[\vr_0,\vu_0,\vt_0,\vB_0]} \ \mbox{ for a.e. } x\in \Omega.
\end{align} 

Then 
\begin{align*}
\Nu=\delta_{[r(t,x),\vU(t,x),\Theta(t,x),\vH(t,x),\mathbb{D}(\vu)(t,x),\Grad\Theta(t,x),\Curl\vH(t,x)]} \ \mbox{ for a.e. } (t,x)\in (0,T)\times\Omega.
\end{align*}
\end{Theorem}
\begin{proof}
Note that \eqref{cr_aux2}, \eqref{cr_aux3} with \eqref{temp} remain valid under the assumptions of Theorem~\ref{T:ur}. Hence it is enough to control the ``residual'' integral, 
\begin{align*}
&\int_0^\tau\intO{\lang\Nu;\left[\vt^2+\vt+|p(\vr,\vt)|+|\vu-\vU|+\vr|s(\vr,\vt)| |\vu| +|\vB||\vu-\vU|\right]_{\rm res}\rang }\dt .
\end{align*} 
Firstly, exactly as before, we get
\begin{align*}
&\int_0^\tau\intO{\lang\Nu;\left[|\vu-\vU|+|\vB||\vu-\vU|\right]_{\rm res}\rang }\dt\\
& \leq  c_\varepsilon\int_0^\tau\Enu(t)\dt+2 \varepsilon\int_0^\tau\intO{\lang\Nu;\left|\mathbb{T}[\Du]-\mathbb{T}[\mathbb{D}(\vU)]\right|^2\rang }\dt .
\end{align*}
Since $p(\vr,\vt) \aleq \vr e(\vr,\vt),$ cf. \cite[Chapter 3]{FeiNoSL}, and   
\begin{align}\label{t_eR}
a\vt^2 =\vr e_R(\vr,\vt),
\end{align}
 we get
\begin{align*}
&\int_0^\tau\intO{\lang\Nu;\left[\vt^2+\vt+|p(\vr,\vt)|\right]_{\rm res}\rang }\dt \leq c\int_0^\tau\intO{\lang\Nu;\left[1+\vr e(\vr,\vt)\right]_{\rm res}\rang }\dt.
\end{align*}
To handle the last part of the ``residual'' integral we note that
$$\vr|s(\vr,\vt)| |\vu| \leq \vr|s_R(\vr,\vt)| |\vu| + \vr|s_M(\vr,\vt)| |\vu| 
\leq c \left(\vt|\vu-\vU|+\vt+\vr|\vu|^2+\vr|s_M(\vr,\vt)|^2\right).$$
Hence the only term we need to control is 
\begin{align*}
&\int_0^\tau\intO{\lang\Nu;\left[\vt|\vu-\vU|+\vr|s_M(\vr,\vt)|^2\right]_{\rm res}\rang }.
\end{align*}
The first term can be handled by the H\" older and the Korn-Poincar\' e inequalities together with \eqref{t_eR}, while for the second term we have 
\begin{align*}
&\int_0^\tau\intO{\lang\Nu;\left[\vr|s_M(\vr,\vt)|^2 \right]_{\rm res}\rang }\dt \leq c\int_0^\tau\intO{\lang\Nu;\left[1+\vr+\vr e_M(\vr,\vt) \right]_{\rm res}\rang }\dt
\end{align*}
due to the structural assumptions \eqref{w1} - \eqref{w7}. See \cite[Section 3.5.1]{BreFei} for the proof.
Finally,  the ``residual'' integral is controlled by 
\begin{align*}
c_\varepsilon\int_0^\tau\mathcal{H}(t)\dt+ \varepsilon\int_0^\tau\intO{\lang\Nu;\left|\mathbb{T}[\Du]-\mathbb{T}[\mathbb{D}(\vU)]\right|^2\rang }\dt ,
\end{align*} 
and the Gronwall lemma yields the result.
\end{proof}
\begin{Remark}
Augmenting the pressure by a radiation component $p_R$ is not only technically convenient but also relevant to problems in astrophysics, cf.  \cite{BATT}. The choice \eqref{pR} is relevant for models of neuron stars studied in \cite{Latt}. 
\end{Remark}

\section*{Conclusions}

In the present paper we have studied the full three-dimensional magnetohydrodynamic model governing the behaviour of a general viscous, compressible, electrically and heat conducting fluid driven by non-conservative boundary conditions. Firstly, we have given a precise definition of a dissipative measure-valued solution (Definition~\ref{D:mvs}). In order to prove  stability of strong solutions in the class of DMV solutions, we have derived a suitable form of the relative energy inequality (Lemma~\ref{L:rei}). This has allowed us to show a few versions of the DMV-strong uniqueness principle based on assumptions we made. The conditional results include the case of bounded DMV solution (Theorem~\ref{T:c1}), the case of constant viscosity, heat conductivity, and magnetic diffusivity coefficients (Theorem~\ref{T:c2}), and finally the case of Boyle's pressure law (Theorem~\ref{T:c3}). 
Finally, the unconditional result with no additional assumptions on DMV solution has been proven (Theorem~\ref{T:ur}).

The results we presented are crucial for a rigorous numerical analysis of the full compressible MHD system in view of the approach developed for compressible fluid flow models in \cite{FLMS}. Indeed, having DMV-strong uniqueness principle at hand,  it is enough to provide a sequence of stable and consistent approximations generating a DMV solution as defined in the present paper in order to   prove convergence of a numerical approximation of the MHD system towards a strong solution of the limit system.

\end{document}